\documentclass{amsart}

\usepackage{amssymb,latexsym,amsmath,amsfonts,graphicx,enumerate}
\usepackage{pict2e}
\usepackage[textsize=footnotesize]{todonotes}
\usepackage{color,comment}
\usepackage{relsize}
\usepackage{hyperref}
\hypersetup{
    colorlinks,
    citecolor=red,
    filecolor=red,
    linkcolor=red,
    urlcolor=red
}
\hypersetup{linktocpage}
\usepackage{xcolor}
\usepackage[all]{xy}
\usepackage{multicol}
\usepackage{multirow}
\usepackage{array}
\usepackage{booktabs}
\usepackage{lscape}
\usepackage[mathscr]{euscript}
\usepackage{enumerate}

\def\ot{\otimes}
\def\BiDer{{\sf Der_{\sf Hopf}}}
\def\Der{{\sf Der}}
\def\Aut{{\sf Aut}}

\newbox\pullbackbox
\setbox\pullbackbox=\hbox{\xy 0;<1mm,0mm>: \POS(4,0)\ar@{-} (0,0) \ar@{-} (4,4)
\endxy}
\def\pullback{\copy\pullbackbox}
\def\Gg{\mathcal G}
\def\Pp{\mathcal P}
\def\Qq{\mathcal Q}
\def\ol{\overline}

\def\Hopf{\mathbf{Hopf}}

\newtheorem{theorem}{Theorem}[subsection] 

\newtheorem{lemma}[theorem]{Lemma}
\newtheorem{proposition}[theorem]{Proposition}

\theoremstyle{definition}

\newtheorem{definition}[theorem]{Definition}

\newtheorem{examples}[theorem]{Examples}
\newtheorem{Assumptions}[theorem]{Assumptions}

\numberwithin{equation}{section}
\begin{document}

\title{Split extension classifiers in the category of cocommutative Hopf algebras}
\author{Marino Gran}
\author{Gabriel Kadjo}
\address[Marino Gran, Gabriel Kadjo]{Institut de Recherche en Math\'ematique et Physique, Universit\'e catholique de Louvain, Chemin du Cyclotron 2, 1348 Louvain-la-Neuve, Belgium}
\email{marino.gran@uclouvain.be}          
\email{gabriel.kadjo@gmail.com}       

\author{Joost Vercruysse}
\address[Joost Vercruysse]{D\'epartement de Math\'ematique, Universit\'e Libre de Bruxelles, Boulevard du Triomphe, 1050 Bruxelles, Belgium.}
\email{jvercruy@ulb.ac.be}

\begin{abstract}
We describe the split extension classifiers in the semi-abelian category of cocommutative Hopf algebras over an algebraically closed field of characteristic zero. The categorical notions of centralizer and of center in the category of cocommutative Hopf algebras is then explored. We show that the categorical notion of center coincides with the one that is considered in the theory of general Hopf algebras.
\end{abstract}

\date{\today}
\maketitle

\section{Introduction}

It is well-known that groups, Lie algebras and Hopf algebras are closely related algebraic structures. 
Recently, we added a new feature to these relationships \cite{GKV}, by observing that the category $\mathbf{Hopf}_{K,coc}$ of cocommutative Hopf algebras over a field of characteristic zero is {\em semi-abelian}  {\cite{JMT}} in the sense of Janelidze, M\'arki and Tholen.  This fact opens the way to many new applications of a wide range of results obtained in that general context, results having consequences in non-abelian homological algebra, radical theory and commutator theory, for instance. We refer the reader to \cite{Vespa} for some other categorical properties of $\mathbf{Hopf}_{K,coc}$, and to \cite{GMVdL} for a conceptual characterization of cocommutative Hopf algebras among cocommutative bialgebras.

One key feature of semi-abelian categories is the concept of an {\em internal action}, which is available without any additional requirements on the category in question. These internal actions behave ``as expected'', in the sense that like for groups or Lie algebras, the induced concept of semi-direct product gives an equivalence between actions and split extensions.

In some cases, but not all, these internal actions are moreover {\em representable}. 
To understand what this means, let us first consider two examples.
An internal action of a group $B$ on a group $G$ corresponds to a group homomorphism $B \to \Aut(G)$ from $B$ to the automorphism group  $\Aut(G)$ of $G$. Similarly, an internal action of a Lie algebra $L_1$ on a Lie algebra $L$ is completely determined by a Lie algebra morphism $L_1 \to \Der(L)$ from $L_1$ to the Lie algebra of derivations of $L$ by $L$. 

A general pointed protomodular category $\mathbf C$ \cite{protomodularity}, 
is said to have representable object actions, or to be \emph{action representable}, in the sense of \cite{BJK},
if given any object $X$ in $\mathbf C$, 
there is a split extension  
\begin{equation}\label{splitgeneral}
\xymatrixcolsep{2pc}\xymatrix{
0 \ar[r]^-{} & X  \ar[r]^{i_1}  & \overline{X}  \ar@<-0.5ex>[r]_{p_2} & [X] \ar@<-0.5ex>[l]_{i_2} \ar[r] & 0 }\end{equation}
with kernel $X$
satisfying the following universal property: for any other split extension with kernel $X$ in $\mathbf C$
\begin{equation}
\xymatrixcolsep{2pc}\xymatrix{
 0 \ar[r] & X  \ar[r]^k  & A \ar@<-0.5ex>[r]_-{f} & B \ar@<-0.5ex>[l]_-{s} \ar[r] & 0 }\end{equation}
there is a unique (up to isomorphism) $\chi \colon B \rightarrow [X]$ (and $\overline{\chi}\colon A\longrightarrow \overline{X}$) such that the following diagram of split exact sequences commutes \[\xymatrixcolsep{2pc}\xymatrix{
0 \ar[r]^-{}  & X \ar@{=}[d] \ar[r]^-{k}  & A \ar@<-0.5ex>[r]_-{f} \ar[d]^{\overline{\chi}} & B \ar@<-0.5ex>[l]_-{s} \ar[r] \ar[d]^{\chi} & 0
 \\
0 \ar[r]^-{} & X  \ar[r]^-{i_1}  & \overline{X} \ar@<-0.5ex>[r]_-{p_2} & [X] \ar@<-0.5ex>[l]_-{i_2} \ar[r] & 0.
}\] 
The object $[X]$ in \eqref{splitgeneral} is called a \textit{split extension classifier} for $X$. Besides the categories of groups and of Lie algebras, there are {some other examples} of categories having this property, {such as the ones of} crossed modules, boolean rings, and commutative von Neumann regular rings \cite{BJK2, BJK}. It is well known that not all semi-abelian categories are action representable: for instance, this is the case for the category of not necessarily unitary rings. When a semi-abelian category $\mathbf C$ is action representable, commutator theory becomes simpler, since the Huq centrality of normal monomorphisms and the Smith centrality of equivalence relations always coincide in this case \cite{MR2553540}. 

In \cite{GKV} we showed that the category $\mathbf{Hopf}_{K,coc}$ of cocommutative Hopf algebras over an algebraically closed field of characteristic zero is 
action representable. In fact, it is widely known that the category of cocommutative Hopf algebras over any commutative base ring $K$ (even not necessarily a field) is isomorphic to the category of internal groups in the category of cocommutative coalgebras over $K$. Hence, even if in this more general case it is unknown whether $\mathbf{Hopf}_{K,coc}$ is semi-abelian, the category is still protomodular and 
it makes sense to consider split extension classifiers in $\mathbf{Hopf}_{K,coc}$. Since the category of cocommutative coalgebras is moreover complete and cartesian closed by Theorem 5.3 in \cite{Barr3}, it follows by Theorem 2.1 in \cite{Clementino} that $\mathbf{Hopf}_{K,coc}$ is action representable over any commutative ring $K$.

However, these arguments do not provide an explicit description of the split extension classifiers in $\mathbf{Hopf}_{K,coc}$ which, for a given object $X$, consists of the object $[X]$ together with a canonical action on $X$ satisfying the above-mentioned universal property. 
In the classical literature on Hopf algebra theory, which, despite its close relationship with groups and Lie algebras, has been developed independently from the evolutions in semi-abelian category theory, 
 similar properties as the one described above have been observed.  Indeed, Sweedler's {\em measuring coalgebras} and {\em measuring bialgebras} \cite{Sweedler} show the existence of universal acting bialgebras on arbitrary coalgebras. Similarly, universal coacting bialgebras and Hopf algebras on certain classes of algebras, which are thought of as describing the `hidden symmetries' of non-commutative spaces, are shown to exist and have been explicitly described \cite{Manin}, \cite{Tambara}. Again, although it is often not so hard to show that this kind of universal (co)acting Hopf algebras exist, the {difficult} part is to give useful descriptions of these objects. The present paper can be understood as a contribution to this work, by joining the forces of Hopf algebra theory with more recent concepts introduced and investigated in categorical algebra.

To be more precise, the aim of the present article is to give an explicit description of the split extension classifiers in the category $\mathbf{Hopf}_{K,coc}$ of cocommutative Hopf algebras over an algebraically closed field of characteristic zero. Our main result is Theorem \ref{universal}, that uses an explicit description of split exact sequences in $\mathbf{Hopf}_{K,coc}$ in terms of semi-direct products (also called smash products) of cocommutative Hopf algebras.
We then analyse the abstract notions of center and of centralizer, introduced by D.\ Bourn and G.\ Janelidze in \cite{MR2553540}, in the category $\mathbf{Hopf}_{K,coc}$. This part {is also based on} some nice results concerning centralizers due to A.\ Cigoli and S.\ Mantovani in \cite{Cigoli}.
We finally compare our description of the center in $\mathbf{Hopf}_{K,coc}$ with the defintion given by A.\ Chirvasitu and P.\ Kasprzak in \cite{Center} for arbitrary (not necessarily cocommutative) Hopf algebras, and observe that they coincide in the cocommutative case.
Since our results heavily rely on the Cartier-Gabriel-Konstant-Milnor-Moore decomposition theorem for cocommutative Hopf algebras which is only valid over an algebraically closed field of characteristic zero, our description only holds in this case. It remains an open question if this description can be extended to the general case.

\noindent {\bf Acknowledgement.} The authors are grateful to Tim Van der Linden for an important remark on a preliminary version of this paper, {and to the referee for some useful suggestions}.

\section{Preliminaries}
\subsection{Action representability for Groups and Lie algebras}

Let us start by recalling how the categories of Groups and Lie algebras satisfy the definition of action representable category, as recalled in the introduction. 

Given a group $G$ and its automorphism group $\Aut(G)$, consider the canonical split extension \begin{equation}\label{splitGrp}
\xymatrixcolsep{2pc}\xymatrix{
0 \ar[r]^-{} & G  \ar[r]^-{i_1}  & G \rtimes \Aut(G) \ar@<-0.5ex>[r]_-{p_2} & \Aut(G) \ar@<-0.5ex>[l]_-{i_2} \ar[r] & 0 }\end{equation}
where the action of $\Aut(G)$ on $G$ defining the semidirect product $G \rtimes \Aut(G)$ is simply given by the evaluation, $p_2$ is the second projection and $i_1, i_2$ are the canonical injections. As explained in \cite{classi, BJK2, BJK} this split extension has a remarkable universal property in the category $\mathbf{Grp}$ of groups:
 any other split extension with kernel $G$ \begin{equation}\label{other} \xymatrixcolsep{2pc}\xymatrix{
0 \ar[r] & G  \ar[r]^k  & A \ar@<-0.5ex>[r]_-{f} & B \ar@<-0.5ex>[l]_-{s} \ar[r] & 0 }\end{equation}
 can be uniquely obtained (up to isomorphism) by pulling back the split extension \eqref{splitGrp}
along a group homomorphism $\chi \colon B \rightarrow \Aut(G)$:
\[\xymatrixcolsep{2pc}\xymatrix{
0 \ar[r] & G  \ar@{=}[d] \ar[r]^k  & A \ar[d]^{\overline{\chi}} \ar@<-0.5ex>[r]_-{f} & B \ar[d]^{\chi} \ar@<-0.5ex>[l]_-{s} \ar[r] & 0 \\
0 \ar[r]^-{} & G  \ar[r]^-{i_1}  & G \rtimes \Aut(G) \ar@<-0.5ex>[r]_-{p_2} & \Aut(G) \ar@<-0.5ex>[l]_-{i_2} \ar[r] & 0}
\] The group action $\chi \colon B \rightarrow \Aut(G)$ of $B$ on $G$ corresponds to the semi-direct product $G \rtimes_{\chi} B$ defining a split extension of $B$ by $G$ isomorphic to \eqref{other}.

In the theory of Lie algebras (over a field $K$) a similar role is played by the split extension 
\begin{equation}\label{usplitLie}
\xymatrixcolsep{2pc}\xymatrix{
0 \ar[r]^-{} & L  \ar[r]^-{i_1}  & L \rtimes_{} \Der(L) \ar@<-0.5ex>[r]_-{p_2} & \Der(L) \ar@<-0.5ex>[l]_-{i_2} \ar[r] & 0.}\end{equation}
Here $L$ is a Lie algebra (over $K$), $\Der(L)$ the Lie algebra of derivations of $L$ by $L$, and 
 $L \rtimes_{} \Der(L)$ the semi-direct product in $\mathbf{Lie}_{K}$, with Lie bracket defined by
$$\big[(x,\phi),(x', \psi ) \big] : =  \Big([x,x'] + \phi(x') - \psi (x), \phi \circ \psi - \psi \circ \phi \Big),$$
for any $x, x' \in L$, for any $\phi, \psi \in \Der(L)$.
In the category $\mathbf{Lie}_{K}$ of Lie algebras, the actions $\rho \colon L_1 \rightarrow \Der(L)$ of $L_1$ on $L$ correspond again to the isomorphism classes $\mathbf{SplExt} (L_1, L)$ of split extensions of $L_1$ by $L$.

\subsection{The category of cocommutative Hopf algebras}

Throughout this paper, let $K$ be a field (although for most definitions it is enough to assume that $K$ is a commutative ring). By a $K$-coalgebra we mean a coassociative and counital coalgebra over $K$, that is a vector space $C$ endowed with linear maps $\Delta:C\to C\ot C$ and $\epsilon:C\to K$ satisfying $(id\ot \Delta)\circ \Delta=(\Delta\ot id)\circ\Delta$ (coassociativity) and $(id\ot \epsilon)\circ \Delta=id=(\epsilon\ot id)\circ\Delta)$ (counitality). We will use the classical Sweedler notation for calculations with the comultiplication, that is, for any $c\in C$ we write $\Delta(c)=c_1\ot c_2$ (summation understood). Coassociativity can then be expressed by the formula
$$c_1\ot c_{2,1}\ot c_{2,2}=c_{1,1}\ot c_{1,2}\ot c_2=c_1\ot c_2\ot c_3.$$
Recall that a $K$-bialgebra $(H,M,u,\Delta,\epsilon)$ is an algebra $H$ with multiplication $M:H\ot H\to H$ and unit $u:K\to H$ that is at the same time a coalgebra with comultiplication $\Delta:H\to H\ot H$ and counit $\epsilon: H\to K$ such that $M$ and $u$ are coalgebra morphisms. A Hopf $K$-algebra is a sextuple $(H,M,u,\Delta,\epsilon,S)$ where $(H,M,u,\Delta,\epsilon)$ is a bialgebra and $S:H\to H$ is a linear map, called the {\em antipode}, that makes the following two diagrams commute
\[\xymatrix{
& H\otimes H \ar@<0.5ex>[rr]^-{S \otimes id} \ar@<-0.5ex>[rr]_-{id \otimes S} & &  H\otimes H \ar[dr]^-{M}  \\
H \ar[rr]_-{\epsilon} \ar[ur]^-{\Delta} & &K \ar[rr]_-{u}& & H
   }\]
or, in Sweedler notation,
$$h_1S(h_2)=\epsilon(h)1_H=S(h_1)h_2,$$
for all $h\in H$.

A Hopf algebra $H$ is said to be \textit{cocommutative} if its underlying coalgebra is cocommutative, this means that the comultiplication map $\Delta$ satisfies $\sigma \circ \Delta = \Delta$, where $\sigma\colon H \otimes H \longrightarrow H \otimes H$ is the switch map $\sigma(x \otimes y)=(y \otimes x)$, for any $x\otimes y\in H \otimes H$.

A morphism of Hopf algebras is a linear map that is both an algebra and a coalgebra morphism (under these conditions the antipode is automatically preserved). 

We denote by $\mathbf{Hopf}_{K}$ the category whose objects are Hopf $K$-algebras and morphisms are morphisms of Hopf $K$-algebras. The full subcategory of cocommutative Hopf $K$-algebras is denoted by $\mathbf{Hopf}_{K,coc}$.

{The categorical concepts of {\em subobject} and {\em kernel} lead to the following notions in the category of Hopf algebras.}
A \textit{sub-Hopf algebra} $H$ of a Hopf algebra $A$ is a subalgebra of $A$ (i.e.\ $u(K)\subset H$ and $M(H\ot H)\subset H$) which is at the same time a subcoalgebra of $A$ (i.e.\ $\Delta(H)\subset H\otimes H$) and which 
is stable under the antipode (i.e.\ $S(H)\subset H$). A sub-Hopf algebra $H$ of $A$ is said to be \textit{normal} (see \cite{Montgom}, for instance) if $a_{1}hS(a_{2}) \in H$ and $S(a_{1})ha_{2} \in H$, $\forall h \in H$ and $\forall a\in A$.

An element $x$ of a Hopf $K$-algebra $H$ is called {\em grouplike} if $\Delta(x)=x\ot x$ and $\epsilon(x)=1$.  A Hopf algebra $H$ is called a {\em group Hopf algebra} if it is generated as a vector space by grouplike elements. The full replete subcategory of ${\bf Hopf}_K$ whose objects are group Hopf algebras is denoted by ${\bf GrpHopf}_K$.
 The set of all grouplike elements of a given Hopf algebra $H$ is denoted by $\Gg(H)$ and the multiplication of $H$ induces a group structure on $\Gg(H)$; the inverse of an element $x\in \Gg(H)$ is given by $S(x)$. This leads to a functor $\Gg:{\bf Hopf}_K\to {\bf Grp}$, we also denote $\Gg(H)=G_H$ for a given Hopf algebra $H$. Moreover, this functor $\Gg$ has a left adjoint $K[-]:{\bf Grp}\to {\bf Hopf}_K$, which assigns to a group $G$ the group algebra $K[G]$ endowed with a comultiplication that turns all elements of $G$ into grouplike elements. If $K$ is a field, then the above adjunction induces an equivalence of categories between ${\bf GrpHopf}_K$ and ${\bf Grp}$.

An element $x$ of a Hopf $K$-algebra $H$ is called {\em primitive} if $\Delta(x)=1\ot x+x\ot 1$ (remark that in this case $\epsilon(x)=0$). A Hopf algebra is called a {\em primitive Hopf algebra} if it is generated as an algebra by primitive elements. The full replete subcategory of ${\bf Hopf}_K$ whose objects are primitive Hopf algebras is denoted by ${\bf PrimHopf}_K$. The set of all primitive elements of a given Hopf algebra $H$ is denoted by $\Pp(H)$ and the commutator bracket (i.e. $[x,y]=yx-xy$) induces a Lie algebra structure on $\Pp(H)$. This leads to a functor $\Pp: {\bf Hopf}_K\to {\bf LieAlg}_K$, where ${\bf LieAlg}_K$ denotes the category of Lie $K$-algebras and Lie algebra morphisms. We also denote $\Pp(H)=L_H$ for a given Hopf algebra $H$. Moreover, this functor has a left adjoint $U:{\bf LieAlg}_K\to {\bf Hopf}_K$, which associates to each Lie algebra the universal enveloping algebra $U(L)$. Since $U(L)$ is generated as an algebra by the elements of $L$, and these are primitive elements in the Hopf algebra $U(L)$, the image of the functor $U$ lies in ${\bf PrimHopf}_K$.
In the case where $K$ is a field of characteristic $0$, the above adjunction induces an equivalence between the category $\mathbf{PrimHopf}_{K}$ and the category $\mathbf{LieAlg}_{K}$, see \cite[Theorem 5.18]{Milnor}.

By construction both ${\bf GrpHopf}_K$ and $\mathbf{PrimHopf}_{K}$ are full and replete subcategories of $\Hopf_{K,coc}$ and the above adjunctions can be viewed as adjunctions between the category $\Hopf_{K,coc}$ and the categories ${\bf Grp}$ and  $\mathbf{LieAlg}_{K}$ respectively.

The category ${\Hopf}_{K,coc}$ is protomodular \cite{protomodularity}, which can be most easily seen by viewing it as the category of internal groups in the category of cocommutative coalgebras. In particular, the zero object in this category is given by the base field $K$, and the zero morphism between Hopf algebras $A$ and $B$ is the morphism $u_B\circ \epsilon_A$. Furthermore, in this category, monomorphisms are exactly the injective morphisms and the classes of normal epimorphisms, regular epimorphisms, epimorphisms and surjective morphisms coincide (see \cite{GKV}).

A {\em short exact sequence} in the category ${\Hopf}_{K,coc}$ is a sequence of the form 
\begin{equation}\label{sequence}
\xymatrix{ 0\ar[rr] && A \ar[rr]^k && H \ar[rr]^p && B \ar[rr] && 0}
\end{equation}
where $0$ represents the zero object, i.e.\ the base field $K$, $k$ and $p$ are Hopf algebra morphisms such that $k$ is the kernel of $p$ and $p$ is the cokernel of $k$.
Recall that a kernel of a morphism $p:H\to B$ in the category of cocommutative Hopf algebras can be computed explicitly by means of the following subobject  of $H$
$$HKer(p)=\{h\in H~|~p(h_1)\ot h_2=1_B\ot h\}.$$
By abuse of notation, we will sometimes denote as well $HKer(p)$ for the kernel morphism $HKer(p)\to H$ in ${\Hopf}_{K,coc}$.
Remark that an expression of the form \eqref{sequence} is an exact sequence in ${\Hopf}_{K,coc}$, whenever $p$ is a surjective homomorphism and $k$ is the kernel of $p$.
A short exact sequence is said to be {\em split} if there exists moreover a Hopf algebra morphism $s:B\to H$ such that $p\circ s=id_B$.

\subsection{Actions in the category of cocommutative Hopf algebras}
Let us recall some useful results from \cite{Molnar}.

\begin{definition}\label{axioms}
Let $B$ be a cocommutative Hopf algebra. A {\em $B$-module Hopf algebra} is a Hopf algebra $A$ that is at the same time a left $B$-module, with action denoted by $\rho\colon B\ot A\to A, \rho(b\ot a)=b\cdot a$, such that its bialgebra maps (multiplication, unit, comultiplication and counit) are maps of $B$-modules. Explicitly, the linear map $\rho$ has to satisfy the following axioms for any $b\in B$ and any $a\in A$:
\begin{enumerate}[({Axiom} 1)]
\item
$1_{B}\cdot a = a$;
\item
$(bb')\cdot a = b\cdot (b'\cdot a)$;
\item
$b\cdot (aa')=(b_{1}\cdot a)(b_{2}\cdot a')$;
\item 
$b\cdot 1_{A}=\epsilon(b)1_{A}$; 
\item
$(b\cdot a)_{1}\otimes (b\cdot a)_{2}=b_{1}\cdot a_{1} \otimes b_{2}\cdot a_{2}$;
\item
$\epsilon(b\cdot a)=\epsilon(b)\epsilon(a)$. \end{enumerate}

If also $A$ is a cocommutative Hopf algebra, we then say that $\rho$ is an {\em action} of $B$ on $A$ in ${\Hopf}_{K,coc}$.
\end{definition}

Remark that $A$ is a $B$-module Hopf algebra if and only if $A$ is an internal Hopf algebra in the symmetric monoidal category of $B$-modules.

\begin{proposition}\cite{Molnar}
Let $B$ be a cocommutative Hopf algebra and $A$ a $B$-module Hopf algebra with action $\rho:B\ot A\to A$. 
Then there exists a Hopf algebra $A \rtimes_{\rho} B$, whose underlying vector space
is $A\ot B$ and whose structure maps are given by
\begin{eqnarray*}
u_{A\rtimes B}&=&u_A\ot u_B\\
M_{A \rtimes_{\rho} B}(a\ot b\ot a'\ot b')&=&a(b_1\cdot a')\ot b_2b'\\
\Delta_{A\rtimes B}&=& (id_A \otimes \sigma \otimes id_{B})(\Delta_{A} \otimes \Delta_{B})\\
\epsilon_{A\rtimes B}&=&\epsilon_A\ot \epsilon_B\\
S_{A\rtimes B}(a\ot b) &=& (S_A(a)\ot 1_B)(1_A\ot S_B(b))= S_B(b_1)\cdot S_A(a)\ot S_B(b_2)
\end{eqnarray*}
{If moreover $A$ is cocommutative, then $A\rtimes_\rho B$ is also cocommutative.}

We call $A \rtimes_{\rho} B$ the \textit{semi-direct product of Hopf algebras} (also known as {\em smash product}) of $B$ and $A$.
\end{proposition}

The following useful lemma about split short exact sequences proved in \cite{GKV} is a reformulation of Theorem $4.1$ in \cite{Molnar}. 

\begin{lemma}\label{Molnar}
Every split short exact sequence in $\mathbf{Hopf}_{K,coc}$ 
 \[\xymatrixcolsep{2pc}\xymatrix{
0 \ar[r]^-{} &  A \ar[r]^-{k}  & H \ar@<0.5ex>[r]^-{p} & B  \ar@<0.5ex>[l]^-{s} \ar[r] & 0 }\] is canonically isomorphic to the semi-direct product exact sequence
\[\xymatrixcolsep{2pc}\xymatrix{
0 \ar[r]^-{} &  A \ar[r]^-{i_1}  & A \rtimes B \ar@<0.5ex>[r]^-{p_2} & B  \ar@<0.5ex>[l]^-{i_2} \ar[r] & 0 }\] where $i_1=id_A\ot u_B$, $i_2=u_A\ot id_B$ and $p_2=\epsilon_A\ot id_B$.
\end{lemma}

\begin{examples}\label{examplesaction}
\begin{enumerate}
\item \cite{Molnar} Let $N$ and $M$ be groups and $\tau\colon M \longrightarrow \Aut(N)$ be a group homomorphism defining an action of $M$ on $N$. Then the map \linebreak $M\times N\longrightarrow N$ defined by $(m,n) \mapsto \tau(m)(n)$ induces a coalgebra map $K[M]\otimes K[N] \longrightarrow K[N]$ which makes $K[N]$ into a $K[M]$-module Hopf algebra. Thus $K[N]\rtimes K[M]$ is a Hopf algebra. One actually has that $$K[N]\rtimes K[M] \cong K[N\rtimes_{\tau} M]$$ where $N\rtimes_{\tau} M$ is the usual semi-direct product of groups with the multiplication defined $\forall n_1,n_2\in N, \forall m_1,m_2 \in M$ by $$(n_1,m_1)(n_2,m_2)=(n_1\tau(m_1)(n_2), m_{1}m_{2}).$$

\item \cite{Molnar} Similarly, for Lie algebras $L$ and $M$ and $\nu\colon L\longrightarrow Der(M)$, a Lie algebra homomorphism, the action $L\times M \longrightarrow M$ defined by $(l,m) \mapsto  \nu(l)(m)$ (for any $(l,m) \in L \times M$) induces a coalgebra map $U(L)\otimes U(M)\longrightarrow U(M)$ making $U(M)$ into a $U(L)$-module Hopf algebra, and as above the Hopf algebra $$U(M)\rtimes U(L) \cong U(M\rtimes_{\nu} L)$$ where $M\rtimes_{\nu} L$ is the usual semi-direct product of Lie algebra with the Lie bracket defined $ \forall m_1,m_2\in M, \forall l_1,l_2 \in L$ by $$ \big[(m_1,l_1),(m_2,l_2)\big] = \big([m_1,m_2]+\nu(l_1)(m_2) - \nu(l_2)(m_1), [l_1,l_2]\big).$$

\item Let $X$ be a normal Hopf subalgebra of a cocommutative Hopf algebra $H$. Then $X$ has a structure of $H$-module Hopf algebra with action defined by
$$h\cdot x=h_1 x S(h_2),$$
for all $x\in X$ and $h\in H$.

\item Let $H$ be any Hopf algebra. Consider $\Pp(H)=L_H$ the Lie algebra of primitive elements of $H$ and $\Gg(H)=G_H$ the group of grouplike elements of $H$. Then $U(L_H)$ and $K[G_H]$ are cocommutative Hopf algebras and by setting 
$$g\cdot x = gxS(g)=gxg^{-1},$$ 
$\forall x\in L_H$ and $\forall g\in G_H$ one defines an action of $K[G_H]$ on $U(L_H)$. 
Furthermore, the map 
\begin{equation}\label{omega}
\omega_H:U(L_{H})\rtimes K[G_{H}]\to H,\quad \omega_H(x\ot g)=xg
\end{equation}
is a Hopf algebra morphism between the induced semi-direct product Hopf algebra and the original Hopf algebra $H$.
\end{enumerate}
\end{examples}

We use Lemma \ref{Molnar} to reformulate the well-known structure theorem for cocommutative Hopf algebras over an algebraically closed field of characteristic zero in terms of split exact sequences (see for instance \cite{Sweedler}, page 279 in combination with Lemma 8.0.1(c)):

\begin{theorem}[Cartier-Gabriel-Kostant-Milnor-Moore]\label{MM} 
For any cocommutative Hopf K-algebra $H$, over an algebraically closed field $K$ of characteristic $0$, 
the Hopf algebra morphism $\omega_H$ \eqref{omega} defined by $\omega_H(x\ot g)=xg$ (for any $g \in G_H$ and $ x \in L_H$) is an isomorphism:
$$U(L_{H})\rtimes K[G_{H}] \stackrel{\omega_H}{\cong} H.$$ 
Consequently, for each $H$ there exists a canonical split exact sequence of cocommutative Hopf algebras of the following form
 \[\xymatrix{
0 \ar[r]^-{} & U(L_{H})  \ar[r]^-{i_{H}} & H \ar@<0.5ex>[r]^-{p_{H}} & K[G_{H}]  \ar@<0.5ex>[l]^-{s_{H}} \ar[r]^-{} & 0 }\] where $i_{H} = \omega_H \circ i_{1}$, $s_{H} = \omega_H \circ i_{2}$ and $p_{H} = p_{2} \circ \omega_H^{-1}$
with the notations of Lemma \ref{Molnar}.
\end{theorem}

\begin{proposition}\label{Kadjoint}
Over an algebraically closed field $K$ of characteristic $0$, the functor $$K[-] \colon \mathbf{Grp} \rightarrow \mathbf{Hopf}_{K,coc}$$ is both a right and a left adjoint to the functor $$\mathcal G \colon \mathbf{Hopf}_{K,coc} \rightarrow \mathbf{Grp}.$$
\end{proposition}

\begin{proof}
We already remarked that the functor $K[-]$ is a left adjoint to the functor $\mathcal G$ (even in the general case of any commutative base ring $K$)
\[\xymatrixcolsep{6pc}
\xymatrix{
\mathbf{Grp} \ar@<1ex>[r]^-{K[-]}_-{\perp} & \mathbf{Hopf}_{K,coc} \ar@<1ex>[l]^-{\mathcal G}.
}\] 
Indeed, if $G \in \mathbf{Grp}$, by definition all elements of $G$ are grouplike elements in $K[G]$. Hence there is a canonical inclusion $\eta_G:G \longrightarrow \mathcal{G}\big(K[G]\big)$.
One can easily check that $\eta_{G}$ is natural and universal. In case that $K$ is a field, it is well-known that $\eta_G=id_G$ (see e.g. \cite{GKV}), and the functor $K[-]$ is then fully faithful.

When $K$ is an algebraically closed field of characteristic zero, as stated, the functor $K[-]$ is actually also a right adjoint to the functor $\mathcal G$.
\[\xymatrixcolsep{6pc}
\xymatrix{
\mathbf{Grp} \ar@<1ex>[r]^-{K[-]}_-{\top} & \mathbf{Hopf}_{K,coc} \ar@<1ex>[l]^-{\mathcal G}
}\]
To see this, let $H \in \mathbf{Hopf}_{K,coc}$. We define the $H$-component of the unit of this second adjunction $H \longrightarrow K[\mathcal{G}(H)]$ as the morphism $p_{H}$ defined in Theorem \ref{MM}.

We  clearly have that $p_{H}$ is natural in $H$,
and $p_{H}$ is universal since for every morphism $f\colon H \longrightarrow K[G]$ in the following diagram
   \[\xymatrix{
0 \ar[r]^-{} & U(L_{H}) \ar[rd]_-{}  \ar[r]^-{i_{H}} & H \ar@<0.5ex>[r]^-{p_{H}} \ar[rd]_-{f} & K[G_{H}] \ar[d]^-{\exists ! \overline{f}}  \ar@<0.5ex>[l]^-{s_{H}} \ar[r]^-{} & 0 \\
& & 0 \ar[r]_-{} & K[G] & }\] 
we have that $f\circ i_{H} = 0_{U(L_{H}),K[G]}$ since primitive elements are preserved by Hopf algebra morphisms and there are no non-trivial primitive elements in a group Hopf algebra. By the universal property of the cokernel $p_{H}$, we conclude that there exists a unique morphism $\overline{f}\colon  K[G_{H}] \longrightarrow K[G]$ such that the above diagram commutes. Since the functor $K[ - ] \colon \mathbf{Grp} \rightarrow \mathbf{Hopf}_{K,coc}$ is fully faithful (as explained in the first part of the proof), the required universal property of the unit of the adjunction is satisfied.
\end{proof}

\begin{proposition}\label{functorQ}
Over an algebraically closed field $K$ of characteristic $0$, the functor $$U \colon \mathbf{Lie }_K \rightarrow \mathbf{Hopf}_{K,coc}$$ has both a right adjoint $\Pp$ and a left adjoint $\Qq$. 
\end{proposition}

\begin{proof}
It is well-known, and recalled above, that the {\em primitive elements} functor $\Pp$ is a right adjoint for the {\em universal enveloping} functor {$U \colon \mathbf{Lie }_K \rightarrow \mathbf{Hopf}_{K,coc}$}.

To construct the left adjoint $\Qq$, take any cocommutative Hopf algebra $H$ and write $H\cong U(L_H)\rtimes K[G]$. Consider the (set-theoretic) equivalence relation $R$ on $L_H$ given by $(x,y)\in R$ if and only if $g\cdot x=y$ for some $g\in G_H$. One can then construct, in the variety of Lie algebras, the smallest {congruence (i.e. equivalence relation compatible with the Lie algebra operations) $\ol R$ containing $R$}. 
Let us call $\Qq(H)$ the quotient of $L_H$ under the equivalence relation $\overline{R}$, then 
by \cite[Theorem 6.12]{Burris} the canonical Lie algebra morphism $F_1:L_H\to \Qq(H)$ satisfies the property that any morphism of Lie algebras $F:L_H\to L$ will factor through $F_1$ if and only if 
$(x,y)\in \ol R$ implies $F(x)=F(y)$.

Now consider any Lie algebra $L$ and a Hopf algebra morphism $f:H\to L$. Since $f$ is completely determined by the image of the grouplike and primitive elements and $U(L)$ has only one grouplike element (to know, the unit), $f$ is completely determined by the Lie algebra morphism $F:=\Pp(f):L_H\to L$. Furthermore, for any $x\in L_H$ and $g\in G_H$, we find 
\begin{eqnarray*}
f((1\ot g)(x\ot 1))&=&f(1\ot g)f(x\ot 1)=F(x)\\
=f(g\cdot x\ot g)&=&f(g\cdot x\ot 1)f(1\ot g)=F(g\cdot x)
\end{eqnarray*}
Hence we find by the above that $F$ factors through $F_1$ {and, accordingly,} we have shown that 
$${\rm Hom}_{\bf Hopf}(H,U(L))\cong {\rm Hom}_{\bf Lie}(\Qq(H),L),$$
which means that $\Qq$ defines a left adjoint for the universal enveloping functor.
\end{proof}

\subsection{Compatibility condition on semi-direct products}
 We now give an algebraic condition that will be useful to guarantee the existence of a Hopf algebra morphism between semi-direct products.
 
First remark that a semi-direct product of cocommutative Hopf algebras $A\rtimes B$, is generated as an algebra by the images of the Hopf algebra morphisms $i_1:A\to A\rtimes B$ and $i_2:B\to A\rtimes B$. Indeed, it is enough to observe that 
$$(a\ot b)=(a\ot 1_B)(1_A\ot b)$$
in $A\rtimes B$. In fact, this property might be seen as a consequence of \cite[Proposition 11]{protomodularity}, which implies that given any split short exact sequence in a pointed protomodular category
 \[\xymatrixcolsep{2pc}\xymatrix{
0 \ar[r]^-{} &  A \ar[r]^-{k}  & C \ar@<0.5ex>[r]^-{f} & B  \ar@<0.5ex>[l]^-{s} \ar[r] & 0 }\] the pair of morphisms $(k,s)$ is jointly epimorphic.
 
The next propostition shows that the equivalence between split extensions and actions for Hopf algebras is also valid on the morphism-level.
 
\begin{proposition}\label{recombine}
Consider the solid diagram in the category  $\mathbf{Hopf}_{K,coc}$
\[\xymatrixcolsep{5pc}\xymatrix{
0 \ar[r]^-{} &  {H_{1}} \ar[r]^-{i_{1}} \ar[d]_-{f} & H_1 \rtimes H_2  \ar[r]_-{} \ar@{.>}[d]_-{h}  &{H_{2}} \ar[d]^-{g} \ar[r]^-{} \ar@/_/[l]_-{i_{2}}& 0 \\
0 \ar[r]^-{} &  {F_{2}} \ar[r]& F_1 \rtimes F_2 \ar[r]_-{}   & F_2 \ar[r]^-{} \ar@/_/[l]& 0 }\]
where the sequences are split exact. We write $\cdot_{H}$ for the action of $H_{2}$ on $H_{1}$ and $\cdot_{F}$ for the action of $F_{2}$ on $F_{1}$. Then there exists a morphism $h\colon H_1 \rtimes H_2 \rightarrow F_1 \rtimes F_2 $ making the diagram commute if and only if $f$ and $g$ are compatible in the following sense:
\begin{equation*}
 f(y\cdot_{H}x)=g(y)\cdot_{F}f(x) \hspace*{0.25cm} \forall x\in H_{1},  \forall y\in H_{2}.
\end{equation*}
If these equivalent conditions hold, then $h=f \otimes g$.
\end{proposition}
\begin{proof}
First remark that, by the observation just before this Proposition, if a Hopf algebra morphism $h$ as in the diagram exists, then necessarily
$$ h(a\ot b)=h(a\ot 1_{H_2})h(1_{H_1}\ot b)=(f(a)\ot 1_{F_2})(1_{F_2} \ot g(b))=f(a)\ot g(b),$$
for any $a \ot b \in H_1 \ot H_2$. This shows that $h=f\ot g$.

Furthermore, given any coalgebra morphisms $f$ and $g$, the map $f\otimes g$ is a coalgebra morphism since the
coalgebra structure of the semi-direct product of two cocommutative Hopf algebras is given by the tensor product coalgebra (i.e. the categorical product in the category of cocommutative coalgebras).

Finally, $f\otimes g$ is an algebra morphism if and only if for all $x,x'\in H_{1}$, $y,y'\in H_{2}$, we have: 
\begin{eqnarray*}
   (f\otimes g)\big((x\otimes y)(x'\otimes y')\big) &=& \big((f\otimes g)(x\otimes y)\big)\big((f\otimes g)(x'\otimes y')\big).
   \end{eqnarray*}
Let us compute the left and right hand side of this expression, using the explicit form of the multiplication in the semi-direct product and the fact that $f$ and $g$ are Hopf algebra morphisms. We find that
\begin{eqnarray*}
(f\otimes g)\big((x\otimes y)(x'\otimes y')\big) &=&(f\otimes g)\big(x(y_{1}\cdot_{H}x')\otimes y_{2}y'\big)\\
&=&f\big(x(y_{1}\cdot_{H}x')\big)\otimes g(y_{2}y')\\
&=& f(x)f\big(y_{1}\cdot_{H}x'\big)\otimes g(y_{2})g(y')
\end{eqnarray*}
and 
\begin{eqnarray*}
\big((f\otimes g)(x\otimes y)\big)\big((f\otimes g)(x'\otimes y')\big)
&=& \big(f(x)\otimes g(y)\big)\big(f(x')\otimes g(y')\big) \\
&=& f(x)\big(g(y)_{1}\cdot_{F}f(x')\big)\otimes g(y)_{2}g(y')\\
&=& f(x)\big(g(y_{1})\cdot_{F}f(x')\big)\otimes g(y_{2})g(y').
\end{eqnarray*}
Therefore $f\ot g$ is an algebra morphism if and only if
$$f(x)f\big(y_{1}\cdot_{H}x'\big)\otimes g(y_{2})g(y')=f(x)\big(g(y_{1})\cdot_{F}f(x')\big)\otimes g(y_{2})g(y')$$
for all $x,x'\in H_1$ and $y,y'\in H_2$.
If the compatibility condition holds, then clearly this condition is satisfied. Conversely, take $x=1_{H_1}$ and $y'=1_{H_2}$ and apply $id\otimes \epsilon$ to the above identity to obtain the compatibility condition.
\end{proof}

\section{Split extension classifiers in $\mathbf{Hopf}_{K,coc}$}\label{construction}

The aim of this section is to construct the \emph{split extension classifier} $[H]$ for a given cocommutative Hopf algebra $H$ in $\Hopf_{K,coc}$ over an algebraically closed field of characteristic zero $K$. Thanks to the decomposition theorem for these Hopf algebras, we only have to identify the grouplike and primitive elements of $[H]$. It might be not a big surprise that $\Gg([H])$ will turn out to be exactly the group $\Aut_{\sf Hopf}(H)$  of Hopf algebra automorphisms of $H$. The primitive elements of $[H]$ turn out to be what we will call Hopf derivations.

\subsection{Hopf derivations}

\begin{definition}
Let $H$ be a Hopf algebra. A {\em Hopf derivation} on $H$ is a linear endomorphism $\psi$ of $H$ which is at the same time a derivation on $H$, i.e. $\psi$ satisfies the \textit{Leibniz rule} 
$$\psi\circ M = M\circ (\psi\otimes id + id\otimes \psi)$$ 
and a coderivation on $H$, i.e. it satifies the \textit{co-Leibniz rule} 
$$\Delta \circ \psi = (\psi\otimes id + id\otimes \psi)\circ \Delta.$$
The set of all Hopf derivations on $H$ is denoted by $\Der_{\sf Hopf}(H)$.
\end{definition}
Given elements $x,y\in H$, we can express the Leibniz and co-Leibniz rules for a linear map $\psi:H\to H$ respectively as
\begin{eqnarray*}
\psi(xy) &=& \psi(x)y + x\psi(y);\\
\psi(x)_1\ot \psi(x)_2 &=& \psi(x_1)\ot x_2 + x_1\ot \psi(x_2).
\end{eqnarray*}

It is well known that the derivations of an algebra form a Lie algebra by means of the commutator bracket. We start this section with the observation that Hopf derivations also form a Lie algebra.

\begin{lemma}
Let $H$ be a Hopf algebra. Then $\BiDer (H)$ is a Lie algebra for the commutator bracket 
$$[\psi_1,\psi_2] = \psi_1\circ \psi_2 - \psi_2\circ \psi_1 \hspace*{4mm}\forall \psi_1, \psi_2 \in \BiDer(H).$$
\end{lemma}

\begin{proof}
As already mentioned, it is well-known that $[\psi_1,\psi_2]$ is again a derivation on $H$ for all derivations $\psi_1$ and $\psi_2$. 
By duality, it follows that $[\psi_1,\psi_2]$ is a coderivation for all coderivations $\psi_1$ and $\psi_2$. Explicitly, this is shown by the following computation:
\begin{eqnarray*}
\Delta \circ [\psi_1,\psi_2] &=& \Delta \circ (\psi_{1} \circ \psi_{2} - \psi_{2} \circ \psi_{1}) \\
& = & (\Delta \circ \psi_{1}) \circ \psi_{2} - (\Delta \circ \psi_{2}) \circ \psi_{1}\\
&=& (\psi_{1} \otimes id + id\otimes \psi_{1}) \circ (\psi_{2} \otimes id + id\otimes \psi_{2}) \circ \Delta \\
&& - (\psi_{2} \otimes id + id\otimes \psi_{2}) \circ (\psi_{1} \otimes id + id\otimes \psi_{1}) \circ \Delta \\
&=& \big((\psi_{1} \circ \psi_{2})\otimes id + \psi_{1}\otimes \psi_{2} + \psi_{2}\otimes \psi_{1} + id\otimes (\psi_{1} \circ \psi_{2})\big)\circ \Delta \\
&&- \big((\psi_{2} \circ \psi_{1})\otimes id + \psi_{2}\otimes \psi_{1} + \psi_{1}\otimes \psi_{2} + id\otimes (\psi_{2} \circ \psi_{1})\big)\circ \Delta \\
&=&  \big((\psi_{1} \circ \psi_{2})\otimes id + id\otimes (\psi_{1} \circ \psi_{2})\big)\circ \Delta\\
&&- \big((\psi_{2} \circ \psi_{1})\otimes id + id\otimes (\psi_{2} \circ \psi_{1})\big)\circ \Delta\\
&=& \big((\psi_{1} \circ \psi_{2} - \psi_{2} \circ \psi_{1})\otimes id + id \otimes (\psi_{1} \circ \psi_{2} - \psi_{2} \circ \psi_{1})\big)\circ \Delta\\
&=& \big([\psi_1,\psi_2]\otimes id + id \otimes [\psi_1,\psi_2]\big)\circ \Delta.
\end{eqnarray*}
Hence we find that, in particular, $\BiDer(H)$ is a Lie algebra for the commutator bracket.
\end{proof}

The following known result for derivations and coderivations will be useful. 

\begin{lemma}\label{derivationunit}
Let $H$ be a Hopf algebra and $\psi\in\BiDer(H)$. Then $\psi\circ u=0$ and $\epsilon\circ \psi=0$.
\end{lemma}
\begin{proof}
Let $u(1_K)=1_H\in H$ be the unit of $H$. When $\psi$ is a derivation we find that
$$\psi(1_{H}) = \psi(1_{H}1_{H}) = \psi(1_H)1_{H} + 1_{H}\psi(1_H) = \psi(1_H) + \psi(1_H),$$
hence $\psi(1_H)=0$ and $\psi\circ u=0$. Similarly, since $\psi$ is a coderivation, we have for all $z\in H$, 
\begin{eqnarray*}
\epsilon\big(\psi(z)\big) &=& \epsilon\Big(\psi(z)_{1}\epsilon\big(\psi(z)_{2}\big)\Big)\\
&=& \epsilon\Big(\psi(z_{1})\epsilon(z_{2}) + z_{1}\epsilon\big(\psi(z_{2})\big)\Big)\\
&=& \epsilon\Big(\psi(z_{1})\epsilon(z_{2})\Big) + \epsilon\Big(z_{1}\epsilon\big(\psi(z_{2})\big)\Big)\\
&=& \epsilon\big(\psi(z)\big) + \epsilon\big(\psi(z)\big)
\end{eqnarray*}
It follows that $\epsilon\circ \psi(z)=0$.
\end{proof}
\subsection{Construction of an action of $K[\Aut_{\sf Hopf}(H)]$ on $U\big(\BiDer(H)\big)$}

Given a Hopf algebra $H$, denote by $\Aut_{\sf Hopf}(H)$ the group of \textit{Hopf algebra automorphisms} of $H$, i.e. the group of bialgebra automorphisms (= bijective endomorphisms) of $H$.

\begin{lemma}\label{rho}
For any $\phi\in \Aut_{\sf Hopf}(H)$ and $\psi\in \BiDer(H)$, the linear endomorphism 
$\phi\cdot \psi:=\phi \circ \psi \circ \phi^{-1}$ is a Hopf derivation on $H$, i.e. there is a map
$$\rho: \Aut_{\sf Hopf}(H)\times \BiDer(H) \longrightarrow \BiDer(H),\quad \rho (\phi,\psi) := \phi \circ \psi \circ \phi^{-1}.$$ 
\end{lemma}
\begin{proof}
Let us first check that $\phi \circ \psi \circ \phi^{-1}$ is indeed a derivation. 
For any $x,y\in H$ we find
\begin{eqnarray*}
(\phi \circ \psi \circ \phi^{-1}) (xy) & = & (\phi \circ \psi) \big(\phi^{-1}(x)\phi^{-1}(y)\big)\\
&=&\phi\Big(\psi\big(\phi^{-1}(x)\big)\phi^{-1}(y)+\phi^{-1}(x)\psi\big(\phi^{-1}(y)\big)\Big) \\
& = & (\phi \circ \psi \circ \phi^{-1})(x)y+x(\phi \circ \psi \circ \phi^{-1})(y).
\end{eqnarray*}
In the above equalities, we have used the fact that $\phi$ and $\phi^{-1}$ are algebra morphisms, and that $\psi$ is a derivation. 
Again, by duality, it follows that  $\phi \circ \psi \circ \phi^{-1}$ is a coderivation as well. Let us give the explicit proof for sake of clarity:
\begin{eqnarray*}
\Delta \circ (\phi \cdot \psi) & = & \Delta \circ \phi \circ \psi \circ \phi^{-1}\\
&=& (\phi \otimes \phi)\circ \Delta \circ \psi \circ \phi^{-1}\\
&=& (\phi \otimes \phi) \circ (\psi \otimes id + id\otimes \psi)\circ \Delta \circ \phi^{-1}\\
&=& (\phi \otimes \phi) \circ (\psi \otimes id + id\otimes \psi)\circ (\phi^{-1}\otimes \phi^{-1})\circ \Delta \\
&=& \big((\phi \circ \psi \circ \phi^{-1})\otimes id + id\otimes  (\phi \circ \psi \circ \phi^{-1})\big)\circ \Delta\\
&=& \big((\phi \cdot \psi)\otimes id + id \otimes (\phi \cdot \psi)\big)\circ \Delta
\end{eqnarray*}
In the above equalities, we have used the fact that $\phi$ and $\phi^{-1}$ are coalgebra morphisms, and that $\psi$ is a coderivation.
\end{proof}

\begin{lemma}\label{olrho}
The map $\rho$ from Lemma~\ref{rho} induces a linear map
$$\overline{\rho}: K[\Aut_{\sf Hopf}(H)]\otimes U\big(\BiDer(H)\big) \longrightarrow U\big(\BiDer(H)\big)$$
defined by $ \overline{\rho} (\phi \otimes \psi) := \rho(\phi,\psi)= \phi \circ \psi \circ \phi^{-1}$ for $\phi\in \Aut_{\sf Hopf}(H)$ and $\psi\in \BiDer(H)$.
\end{lemma}

\begin{proof}
First remark that the map $\rho:\Aut_{\sf Hopf}(H)\times \BiDer(H) \to \BiDer(H)$ is $K$-linear in the second argument. Therefore, we find by linear extension a bilinear map $K[\Aut_{\sf Hopf}(H)]\times \BiDer(H) \to \BiDer(H)$ which in turn induces a linear map $\rho':K[\Aut_{\sf Hopf}(H)]\ot \BiDer(H) \to \BiDer(H)$.

Fix any $\phi\in\Aut_{\sf Hopf}(H)$. Then the induced morphism $\rho_\phi:\BiDer(H)\to \BiDer(H)$ defined by $\rho_\phi(\psi)=\rho'(\phi\ot\psi)=\rho(\phi,\psi)$ is a Lie algebra morphism. Indeed:
\begin{eqnarray*}
\phi\cdot[\psi, \psi^{'}] &=& \phi \cdot (\psi \circ \psi^{'} - \psi^{'}\circ \psi)\\
&=& \phi\circ \psi\circ \psi^{'}\circ \phi^{-1} - \phi\circ \psi^{'}\circ \psi\circ\phi^{-1}\\
&=& [\phi\circ \psi\circ \phi^{-1}, \phi\circ \psi^{'}\circ \phi^{-1}] \\
&=& [\phi\cdot \psi, \phi\cdot \psi^{'}].
\end{eqnarray*}
Hence $\rho_\phi$ gives rise to an algebra map $\overline\rho_\phi:U\big(\BiDer(H)\big)\to U\big(\BiDer(H)\big)$. We can now define $\overline\rho:K[\Aut_{\sf Hopf}(H)]\otimes U\big(\BiDer(H)\big) \to U\big(\BiDer(H)\big)$ as $\overline\rho(\phi\ot x)=\overline\rho_\phi(x)$ for $\phi\in \Aut_{\sf Hopf}(H)$ and $x\in U\big(\BiDer(H)\big)$.
\end{proof}

\begin{proposition}\label{constructionSEC}
The map $\overline{\rho}$ from Lemma \ref{olrho} defines an action of Hopf algebras of $K[\Aut_{\sf Hopf}(H)]$ on $U\big(\BiDer(H)\big)$. 
Accordingly, this action corresponds to a split exact sequence via the semidirect product
$U\big(\BiDer(H)\big)\rtimes_{\overline{\rho}} K[\Aut_{\sf Hopf}(H)]$, that we will denote by $[H]$.
\end{proposition}

\begin{proof}
We have to verify that $\overline{\rho}$ satisfies the axioms of an action of cocommutative Hopf algebras (Definition \ref{axioms}). Axiom 3 and Axiom 4 follow immediately from the construction of $\overline\rho$ in the proof of Lemma~\ref{olrho}. Indeed, by construction for any $\phi\in K[\Aut_{\sf Hopf}(H)]$, the map $\overline\rho_\phi$ is an algebra morphism, which expresses exactly Axiom 3 and Axiom 4. As a consequence, it is enough to verify the remaining axioms on base elements of the vector space $K[\Aut_{\sf Hopf}(H)]$ (i.e.\ elements $\phi\in \Aut_{\sf Hopf}(H)$) and generators of the algebra $U\big(\BiDer(H)\big)$ (i.e.\ elements of the form $\psi\in\BiDer(H)$).
\begin{itemize}
\item (Axiom 1) $id_{K[\Aut_{\sf Hopf}(H)]}\cdot \psi = \psi$;
\item (Axiom 2) $(\phi \circ \phi^{'})\cdot \psi = \phi\circ\phi^{'}\circ\psi\circ\phi^{'-1}\circ \phi^{-1} =  \phi \cdot(\phi^{'}\cdot \psi)$;
\item  (Axiom 5) Remark that any $\phi\in \Aut_{\sf Hopf}(H)$ is a grouplike element of $K[\Aut_{\sf Hopf}(H)]$ and any $\psi\in \BiDer(H)$ is a primitive element in $U\big(\BiDer(H)\big)$. Since $\phi\cdot \psi\in\BiDer(H)$ this is as well a primitive element in $U\big(\BiDer(H)\big)$. Hence we obtain
$$(\phi \cdot \psi)_{1} \otimes (\phi \cdot \psi)_{2} = \Delta(\phi \cdot \psi) = (\phi \cdot \psi)\otimes id + id\otimes (\phi \cdot \psi)$$ 
and
\begin{eqnarray*}
(\phi_{1}\cdot \psi_{1}) \otimes (\phi_{2}\cdot \psi_{2}) &=& (\phi \cdot \psi) \otimes (\phi \cdot id) + (\phi \cdot id) \otimes (\phi \cdot \psi) \\
&=& (\phi \cdot \psi)\otimes id + id\otimes (\phi \cdot \psi),
\end{eqnarray*}
from which the axiom follows.
\item  (Axiom 6) We have $\epsilon(\phi \cdot \psi) = 0 = \epsilon(\phi)\epsilon(\psi)$ since  $\phi \cdot \psi$ and $\psi$ are primitive elements.
\end{itemize}
\end{proof}

\subsection{Construction of an action of $[H]$ on $H$}

\begin{proposition}\label{actionSEC}
Let $H$ be a Hopf algebra and consider the Hopf algebra $[H]=U\big(\BiDer(H)\big)\rtimes_{\overline{\rho}} K[\Aut_{\sf Hopf}(H)]$ from Proposition~\ref{constructionSEC}.
The map $$\star : [H] \otimes H \longrightarrow H $$ 
given by
$$ (\psi \otimes \phi)\star h := \psi\big(\phi(h)\big),$$
for any $\psi\in\BiDer(H)$ and $\phi\in\Aut_{\sf Hopf}(H)$ turns $H$ into an $[H]$-module Hopf algebra.
\end{proposition}
\begin{proof}
Let us check that $\star$ satisfies the axioms of a module Hopf algebra (Definition \ref{axioms}).
\begin{itemize}
\item (Axiom 1) $\big(id_{U\big(\BiDer(H)\big)}\otimes id_{K[\Aut_{\sf Hopf}(H)]}\big)\star h = h$.
\item (Axiom 2) One has the equalities \begin{eqnarray*}
(\psi \otimes \phi)(\psi^{'} \otimes \phi^{'})\star h &=& \big(\psi\overline{\otimes}(\phi\cdot_{\overline{\rho}}\psi^{'}) \otimes (\phi \circ \phi^{'})\big) (h)\\
&=& (\psi\circ \phi\circ \psi^{'}\circ \phi^{-1} \circ \phi \circ \phi^{'}) (h)\\
&=& (\psi \otimes \phi)\star \big((\psi^{'} \otimes \phi^{'}) \star h\big).
\end{eqnarray*} 
Here we denoted by $\overline\ot$ the multiplication in the universal enveloping algebra $U\big(\BiDer(H)\big)$.
\item (Axiom 3) First remark that for any $\phi\in\Aut_{\sf Hopf}(H)$ and $\psi\in\BiDer(H)$ we have 
\begin{eqnarray*}
\Delta(\psi \otimes \phi) &=& (id \otimes \sigma \otimes id)\big(\Delta(\psi) \otimes \Delta(\phi)\big)\\
&=& (id\otimes \sigma\ot id)((\psi \ot 1+1\ot \psi)\ot(\phi\ot \phi))\\
&=&\psi \otimes \phi \otimes id \otimes \phi + id \otimes \phi \otimes \psi \otimes \phi
\end{eqnarray*}
where we used that $\phi$ is grouplike and $\psi$ is primitive. Then we can now verify the axiom.
\begin{eqnarray*}
(\psi \otimes \phi)\star hh' &=& \psi\big(\phi(hh')\big)\\
&=& \psi\big(\phi(h)\phi(h')\big) \\
&=& \psi (\phi(h))\phi(h') + \phi(h)\psi(\phi(h')),\\
&=& \big((\psi \otimes \phi)\star h)\big((id\otimes \phi)\star h'\big) + \big((id\otimes \phi)\star h)\big((\psi\otimes \phi)\star h'\big)\\
&=&\big( (\psi\ot \phi)_1\star h\big)\big( (\psi\ot \phi)_2\star h'\big)
\end{eqnarray*} 
where we used that $\phi$ is an algebra morphism in the second equality and that $\psi$ is a derivation in the third equality. 
\item (Axiom 4) The right-hand side of the equality is given by $$(\psi \otimes \phi)\star 1_{H} = \psi\big(\phi(1_{H})\big) = \psi(1_{H}) = 0,$$
where we applied Lemma~\ref{derivationunit}.
 The left-hand side of the equality is given by $$\epsilon(\psi \otimes \phi)1_{H} = \epsilon(\psi)\epsilon(\phi)1_{H} = 0$$ where used $\epsilon(\psi)=0$ since $\psi$ is a primitive element of $U\big(\BiDer(H)\big)$.
\item (Axiom 5) The left-hand side of the identity to be checked is given by 
\begin{eqnarray*}
\big((\psi \otimes \phi)\star h\big)_{1} \otimes \big((\psi \otimes \phi)\star h\big)_{2} &=& \Delta\big((\psi \otimes \phi)\star h\big) 
= \Delta \circ \psi \circ \phi (h)\\ 
&=& (\psi \otimes id + id \otimes \psi)\circ \Delta \circ \phi(h) \\
&=& (\psi \otimes id + id \otimes \psi)\circ (\phi \otimes \phi) \circ \Delta(h) \\
&=& \big((\psi \circ \phi)\otimes \phi + \phi \otimes (\psi \circ \phi)\big) \circ \Delta(h). 
\end{eqnarray*}
On the other hand, since 
$$\Delta(\psi \otimes \phi) = \psi \otimes \phi \otimes id \otimes \phi + id \otimes \phi \otimes \psi \otimes \phi,$$ 
the right-hand side of the identity is given by 
\begin{eqnarray*}
\big((\psi \otimes \phi)_{1}\star h_{1}\big) \otimes \big((\psi \otimes \phi)_{2}\star h_{2}\big)&=&
\big((\psi \otimes \phi)\star h_{1}\big) \otimes \big((id\otimes \phi)\star h_{2}\big)\\
 &&+ \big((id\otimes \phi)\star h_{1}\big)\otimes \big((\psi \otimes \phi)\star h_{2}\big)\\
&=& \psi\big(\phi(h_{1})\big) \otimes \phi(h_{2}) + \phi(h_{1})\otimes \psi\big(\phi(h_{2})\big)\\
&=& \big((\psi \circ \phi)\otimes \phi + \phi \otimes (\psi \circ \phi)\big) \circ \Delta(h).
\end{eqnarray*}
\item (Axiom 6) The right-hand side of the identity is given by  $$\epsilon(\psi\otimes \phi)\epsilon(h) = \epsilon(\psi)\epsilon(\phi)\epsilon(h) = 0$$ 
where we used again that $\epsilon(\psi)=0$ since $\psi$ is a primitive element. The left-hand side of the identity is given by 
$$\epsilon\big((\psi \otimes \phi)\star h\big)=\epsilon\circ \psi\circ\phi (h)=0$$ 
where we used Lemma~\ref{derivationunit}.
\end{itemize}
\end{proof}

\subsection{Universal property of the split extension classifier}\label{construction3}

The aim of this section is to prove that for a given cocommutative Hopf algebra $H$, the Hopf algebra $[H]$ constructed in Proposition~\ref{constructionSEC} together with its action on $H$ defined in Proposition~\ref{actionSEC} is exactly the split extension classifier of $H$ in $\Hopf_{K,coc}$, i.e.\ it satisfies the universal property recalled in the introduction. 

\begin{theorem}\label{universal}
For any algebraically closed field $K$ of characteristic \hspace*{0.1mm} $0$, the split extension classifier of any cocommutative Hopf $K$-algebra $H$ is given by 
$$[H]=U\big(\BiDer(H)\big)\rtimes_{\overline{\rho}} K\big[\Aut_{\sf Hopf}(H)\big]$$ where $$\overline{\rho}: K[\Aut_{\sf Hopf}(H)]\otimes U\big(\BiDer(H)\big) \longrightarrow U\big(\BiDer(H)\big)$$ is defined by $$\overline{\rho} (\phi \otimes \psi) := \phi \cdot \psi = \phi \circ \psi \circ \phi^{-1}$$ on the generators of $K[\Aut_{\sf Hopf}(H)]\otimes U\big(\BiDer(H)\big)$.
\end{theorem}

\begin{proof}
Recall from Proposition~\ref{actionSEC} that $[H]$ acts on $H$ by $\star$, and therefore we have the following split exact sequence in $\Hopf_{K,coc}$
\begin{equation}\label{H[H]}
\xymatrix{
0 \ar[r]^-{} &  H \ar[r]^-{i_{1}} & H\rtimes_{\star} [H] \ar@<-0.5ex>[r]_-{p_{2}}   & [H] \ar[r]^-{} \ar@<-0.5ex>[l]_-{i_{2}} & 0 
}
\end{equation}
Consider now any split exact sequence in $\mathbf{Hopf}_{K,coc}$ of the form
\begin{equation}\label{first}
\xymatrixcolsep{2pc}\xymatrix{
 0 \ar[r] & H  \ar[r]  & B \ar@<-0.5ex>[r] & A \ar@<-0.5ex>[l] \ar[r] & 0. }\end{equation} 
 By applying the semi-direct product decomposition of $A$ recalled in Theorem \ref{MM}, and by using Lemma \ref{Molnar}, the split exact sequence \eqref{first} is isomorphic to the following split exact sequence in $\mathbf{Hopf}_{K,coc}$. \begin{equation*}
\xymatrixcolsep{2pc}\xymatrix{
 0 \ar[r] & H  \ar[r]^-{k}  & H\rtimes_{A} \big(U(L_A)\rtimes K[G_A]\big) \ar@<-0.5ex>[r]_-{f} & U(L_A)\rtimes K[G_A] \ar@<-0.5ex>[l]_-{s} \ar[r] & 0. }\end{equation*} 
We have to show that there is a unique morphism of split exact sequences between between \eqref{H[H]} and \eqref{first}. 
To this end, we will first construct a morphism $$\chi\colon U(L_A)\rtimes K[G_A] \longrightarrow U\big(\BiDer(H)\big)\rtimes_{\overline{\rho}} K\big[\Aut_{\sf Hopf}(H)\big]$$ compatible with the morphism $id_{H}$ (i.e. such that the following diagram commutes, see Proposition \ref{recombine}), and then prove the uniqueness of such a morphism $\chi$.
\begin{equation}
\xymatrixcolsep{0.8pc}\xymatrix{
0 \ar[r]^-{} &  H \ar[r]^-{k} \ar@{=}[d] & H\rtimes_{A} \big(U(L_A)\rtimes K[G_A]\big) \ar@<-0.5ex>[r]_-{f} \ar@{.>}[d]_-{id_{H}\otimes \chi}  & U(L_A)\rtimes K[G_A] \ar@{.>}[d]^-{\chi} \ar[r]^-{} \ar@<-0.5ex>[l]_-{s} & 0 \\
0 \ar[r]^-{} &  H \ar[r]^-{i_{1}} & H\rtimes_{\star} \Big(U\big(\BiDer(H\big))\rtimes_{\overline{\rho}} K\big[\Aut_{\sf Hopf}(H)\big]\Big) \ar@<-0.5ex>[r]_-{p_{2}}   & U\big(\BiDer(H)\big)\rtimes_{\overline{\rho}} K\big[\Aut_{\sf Hopf}(H)\big] \ar[r]^-{} \ar@<-0.5ex>[l]_-{i_{2}} & 0 }
\end{equation}

\underline{Step 1: Existence of $\chi$}. \underline{a) Construction of a group morphism $\chi_{G}: G_A \longrightarrow \Aut_{\sf Hopf}(H)$}.\\
If $\rho :A\ot H\to H$ is the action of $A$ on $H$ induced by the split exact sequence \eqref{first}, defined by $\rho(a\ot h)=a\cdot h$, one has a group homomorphism defined by $\chi_{G}: G_A \longrightarrow \Aut_{\sf Hopf}(H): \chi_{G}(g)(h):= g\cdot h$, for any $g\in G_A$ and $h\in H$.

Let us first check that $\forall g\in G_A$, the map $\chi_{G}(g)$ is an endomorphism of Hopf algebras of $H$: 
\begin{itemize}
\item $\chi_{G}(g)$ is an algebra morphism since $\forall h,h'\in H$ we have  
\begin{eqnarray*}
\chi_{G}(g)(hh') &=& g\cdot (hh') = (g_1\cdot h)(g_2\cdot h')\\
 &=& (g\cdot h)(g\cdot h') = \big(\chi_{G}(g)(h)\big)\big(\chi_{G}(g)(h')\big)\\
\chi_G(g)(1_H) &=& g\cdot 1_H = \epsilon(g)1_H = 1_H
\end{eqnarray*}
by axioms 3 and 4 of an action of cocommutative Hopf algebras and by the fact that $g$ is a group-like element of $H$. 
\item $\chi_{G}(g)$ is a coalgebra morphism since for all $h\in H$ we have 
\begin{eqnarray*}
\Delta\big(\chi_{G}(g)(h)\big) &=& \Delta\big(g\cdot h)
= (g\cdot h)_{1}\otimes (g\cdot h)_{2}\\
&=& (g_{1}\cdot h_{1})\otimes (g_{2}\cdot h_{2})
= (g\cdot h_{1})\otimes (g\cdot h_{2})\\
&=& \big(\chi_{G}(g)\otimes \chi_{G}(g)\big)\Delta(h)\\
\epsilon\big(\chi_{G}(g)(h)\big) &=& \epsilon(g\cdot h) = \epsilon(g)\epsilon(h)=\epsilon(h)
\end{eqnarray*} 
by the axioms 5 and 6 of an action of cocommutative Hopf algebras and by the fact that $g$ is a group-like element of $H$. 
\end{itemize}
So far, we have proven that there is a well-defined map $\chi_G:G_A\to {\sf End}_{\sf Hopf}(H)$. Let us show now that $\chi_G$ is moreover a monoid morphism.
\begin{itemize}
\item 
indeed: for all $g,g'\in G$ and $h\in H$, we have
\begin{eqnarray*}
\big(\chi_{G}(g')\circ \chi_{G}(g)\big)(h) &=& g' \cdot (g \cdot h)
{=} (g'g)\cdot h
= \chi_G(g'g)(h)
\end{eqnarray*}
where the second equality follows follows axiom 2. Observe that the map $\chi_G$ preserves the neutral element by axiom 1.
\end{itemize}
In particular, it follows that $\chi_G(g^{-1})=\chi_G(g)^{-1}$ and therefore, $\chi_G(g)$ is an automorphism of Hopf algebras for all $g\in G$, i.e. $\chi_G:G_A\to {\sf Aut}_{\sf Hopf}(H)$ is a group morphism as stated.

\underline{b) Construction of a Lie algebra morphism $\chi_{L}: L_A \longrightarrow \BiDer(H)$}.\\
Consider again the action  of $A$ on $H$, $\rho :A\ot H\to H$, with $\rho(a\ot h)=a\cdot h$ and define $\chi_L: L_A\to {\sf End}_L(H), \chi_L(x)(h)=x\cdot h$ for all $x\in L_A$ and $h\in L$.
Let us prove that $\chi_{L}(x)$ is a Hopf derivation in $H$ for all $x\in L_A$.
\begin{itemize}
\item $\chi_{L}(x)$ is a derivation since we have for all $h,h'\in H$
\begin{eqnarray*}
\chi_{L}(x)(hh') &=& (x_1\cdot h)(x_2\cdot h')\\
&=& (x \cdot h)(1\cdot h') + (1\cdot h)(x\cdot h') \\
&=& (x\cdot h)h'+h(x\cdot h') \\
&=& \chi_{L}(x)(h)h' + h\chi_{L}(x)(h')
\end{eqnarray*}
by axiom 1 and 3 of an action of cocommutative Hopf algebras and by the fact that $x$ is a primitive element. 
\item From duality, we find that $\chi_{L}(x)$ is a coderivation. Explicitly, this follows since we have for all $h\in H$,
\begin{eqnarray*}
\big(\Delta \circ \chi_{L}(x)\big)(h)&=& \Delta(x\cdot h) \\
&=& (x\cdot h)_{1} \otimes (x\cdot h)_{2}  \\
&=& \big(x_{1}\cdot h_{1}\big)\otimes \big(x_{2}\cdot h_{2}\big)\\
&=& \big(x\cdot  h_{1}\big) \otimes \big(1\cdot  h_{2}\big) + \big(1\cdot h_{1}\big)\otimes \big(x\cdot  h_{2}\big)\\
&=& \big(x\cdot  h_{1}\big)\otimes h_{2} + h_{1}\otimes \big(x\cdot h_{2}\big)\\
&=& \chi_{L}(x)(h_{1})\otimes h_{2} + h_{1}\otimes \chi_{L}(x)(h_{2})\\
&=& \big(\chi_{L}(x)\otimes id + id \otimes \chi_{L}(x)\big)\circ \Delta(h)
\end{eqnarray*}
In the above equalities, we have used the axioms 1 and 5 of an action of cocommutative Hopf algebras, and the fact that $x$ is a primitive element.
\end{itemize}
Hence $\chi_L:L_A\to \BiDer(H)$ is well-defined. 
Furthermore, $\forall x,y\in L_A$, $\forall h\in H$, one has: 
\begin{eqnarray*}
\chi_{L}[x,y](h) &=& \chi_{L}(xy-yx)(h)\\
&=& (xy-yx)\cdot  h\\
&=& (xy)\cdot  h - (yx)\cdot  h \\
&=& x\cdot  (y\cdot  h)-y\cdot (x\cdot  h) \\
&=& \big(\chi_{L}(x)\circ \chi_{L}(y)- \chi_{L}(y)\circ \chi_{L}(x)\big)(h) \\
&=& \big[\chi_{L}(x),\chi_{L}(y)\big](h)
\end{eqnarray*}
where we used the axiom 2 of action of cocommutative Hopf algebras. Therefore, $\chi_L:L_A\to \BiDer(H)$ is a morphism of Lie algebras.

\underline{c) Compatibility between group and Lie algebra part}.\\
Let us check that the morphisms $$U(\chi_{L})\colon U(L_A) \longrightarrow U\big(\BiDer(H)\big)$$ and $$K[\chi_{G}]\colon K[G_A] \longrightarrow K\big[\Aut_{\sf Hopf}(H)\big]$$ satisfy the compatibility condition (Proposition \ref{recombine}). We denote the action of $K[G_A]$ on $U(L_A)$ by $\cdot_A$. Recall from Example \ref{examplesaction} (4) that $g\cdot_A x=gxg^{-1}$ for all $\forall g\in G_{A}, \forall x \in L_{A}$. 
Then for all $h\in H$ we find
\begin{eqnarray*}
\chi_{L}(g\cdot_{A} x)(h) &=& (g\cdot_A x)\cdot h = (gxg^{-1})\cdot h \\
&=& g\cdot \big(x\cdot (g^{-1}\cdot h)\big) \\
&=& \big(\chi_{G}(g)\circ \chi_{L}(x)\circ  \chi_{G}(g^{-1})\big)(h)\\
&=& \big(\chi_G(g)\cdot_{\bar\rho} \chi_L(x)\big)(h)
\end{eqnarray*}
where we applied again axiom 2 of a Hopf algebra action. As a consequence we obtain the morphisms $U(\chi_{L})$ and $K[\chi_{G}]$ can be recombined to define a Hopf algebra morphism $\chi := U(\chi_{L}) \otimes K[\chi_{G}]:A\to [H]$.

\underline{d) Compatibility between $id_H$ and $\chi$.}\\
We will now check that also the morphisms $id_{H}$ and $\chi$ satisfy the compatibility condition (Proposition \ref{recombine}). Indeed, for a generator $x\otimes g \in U(L_A)\rtimes_{A} K[G_A]$, with $x\in L_A$ and $g\in G_A$, associated to the element $xg\in A$ by the isomorphism of Cartier-Gabriel-Konstant-Milnor-Moore (Theorem \ref{MM}), and for all $h\in H$, one has:
\begin{eqnarray*}
\Big(\chi(xg)\Big) \cdot_{\star} id_{H}(h) &=& \Big(U(\chi_{L})(x) \otimes K[\chi_{G}](g)\Big)\cdot_{\star} id_{H}(h)\\
&=& \Big(\chi_{L}(x)\circ \chi_{G}(g)\Big)(h)
= x\cdot (g\cdot  h)\\
&=&  (xg)\cdot  h
=  id_{H}\big((xg)\cdot  h\big)
\end{eqnarray*}
by the axiom 2 of an action of cocommutative Hopf algebras.

\underline{Step 2: The uniqueness}.\\
It remains to prove the uniqueness of the morphism $\chi$. For this, consider two other morphisms $\xi$ and $\overline{\xi}$ making the following diagram commute
\begin{equation}
\xymatrixcolsep{0.8pc}\xymatrix{
0 \ar[r]^-{} &  H \ar[r]^-{k} \ar@{=}[d] & H\rtimes_{A} \big(U(L_A)\rtimes K[G_A]\big) \ar@<-0.5ex>[r]_-{f} \ar@{.>}[d]_-{\overline{\xi}}  & U(L_A)\rtimes K[G_A] \ar@{.>}[d]^-{\xi} \ar[r]^-{} \ar@<-0.5ex>[l]_-{s} & 0 \\
0 \ar[r]^-{} &  H \ar[r]^-{i_{1}} & H\rtimes_{\star} \Big(U\big(\BiDer(H\big))\rtimes_{\overline{\rho}} K\big[\Aut_{\sf Hopf}(H)\big]\Big) \ar@<-0.5ex>[r]_-{p_{2}}   & U\big(\BiDer(H)\big)\rtimes_{\overline{\rho}} K\big[\Aut_{\sf Hopf}(H)\big] \ar[r]^-{} \ar@<-0.5ex>[l]_-{i_{2}} & 0 }
\end{equation}

By Proposition \ref{recombine} we know that $\overline{\xi}= {id}_H \otimes \xi$ and ${id}_H$ and $\xi$ are compatible, hence
we obtain for all $a\in A$ and $h\in H$ the following equality in $H$ 
$$\xi(a)\cdot_{\star} h = a\cdot  h = \chi(a)\cdot_{\star}h.$$
The morphism $\xi$ induces a group homomorphism ${\xi}_G \colon G_A \rightarrow \Aut_{\sf Hopf}(H)$ and a Lie algebra homomorphism ${\xi}_L \colon L_A \rightarrow \BiDer(H)$.
We see that for all $x \in L_A$ and $g \in G_A$ $${\xi}_L(x)(h)= {\chi}_L(x)(h),$$ and 
$${\xi}_G(g)(h)= {\chi}_G(g)(h).$$
Therefore, since $A=U(L_A)\rtimes K[G_A]$, we find that $\xi=\chi$.
\end{proof}

\section{Centers and centralizers in $\mathbf{Hopf}_{K,coc}$}
 In the last section we compare the categorical notions of center and centralizer in the semi-abelian category $\mathbf{Hopf}_{K,coc}$ with the ones recently introduced in the context of general Hopf algebras. \\
 
\subsection{Centers and centralizers in semi-abelian action representable categories}\label{general}
Recall that the \textit{center of a group} $G$, denoted by $Z(G)$, is defined by 
\begin{center}
 $Z(G) = \{x\in G, xg=gx \hspace*{0.15cm} \forall g\in G\}.$
 \end{center} More generally, the \textit{centralizer $C_{G}(H)$ of a subgroup} $H$ of a group $G$ is the set of elements of $G$ which commute with every element of $H$, i.e. $$C_{G}(H) = \{x\in G, xh=hx \hspace*{0.15cm} \forall h\in H \}.$$
The center of a group $G$ is the kernel of the conjugation map, i.e. the kernel of the morphism $\phi: G \longrightarrow \Aut(G)$, where $\phi(x)$ is defined by $\phi(x)(y)=xyx^{-1}$, $\forall x,y\in G$. In the same way, the centralizer of a subgroup $H$ of a group $G$ is the kernel of the morphism $\phi: H \longrightarrow \Aut(G)$, where $\phi(h)$ is defined by $\phi(h)(g)=hgh^{-1}$, $\forall h\in H$, $\forall g\in G$.

In the category $\mathbf{Lie}_{K}$ of $K$-Lie algebras the notions of center and of centralizer are defined similarly. The \textit{center of a Lie algebra} $L$, denoted by $Z(L)$, is the ideal
\begin{center}
$Z(L) = \{x\in L, [x,a] = 0 \hspace*{0.15cm} \forall a\in L\}$
\end{center}
whereas the \textit{centralizer $C_{L}(I)$ of a Lie subalgebra $I$ of a Lie algebra $L$} is the ideal of the elements that commute with every element of $I$, i.e. $$C_{L}(I) = \{x\in L, [x,i]=0 \hspace*{0.15cm} \forall i\in I \}.$$
As in the case of groups, also in $\mathbf{Lie}_{K}$ these subalgebras occur as kernels of suitable morphisms.
Indeed, the center of a Lie algebra $L$ is the kernel of the adjoint representation, i.e. of the Lie algebra homomorphism $ad: L \longrightarrow End(L)$ defined by $$ ad(x):= [x,-], \quad \forall x \in L.$$ The centralizer of a Lie subalgebra $I$ of $L$ is the kernel of the homomorphism $$ad: I \longrightarrow End(L)$$ defined by $$ad(l):= [l,-], \quad \forall l \in L.$$

In the context of action representable semi-abelian categories A. Cigoli and S. Mantovani gave a description of the center and of the centralizer of a normal subobject \cite{Cigoli} that is similar to the one just recalled in the categories of groups and of Lie algebras. If $h: H \rightarrow A$ is a normal subobject of $A$, the quotient map will be denoted by $p: A \rightarrow A/H$, whereas $\big(Eq(p), p_1, p_2\big)$ will denote the equivalence relation on $A$ occurring as the kernel pair of $p$, so that $p_1$ and $p_2$ are the projections in the following pullback:
\[\xymatrixcolsep{2pc}\xymatrix{
 Eq(p) \ar[r]^-{p_2} \ar[d]_-{p_1} \ar@{}[rd]|<<<{\pullback}  & A \ar[d]^-{p} \\
A \ar[r]_-{p}  & A/H  }\] 
We write $\Delta: A \longrightarrow Eq(p)$ for the unique morphism such that $p_{1}\circ \Delta = 1_{A} = p_{2}\circ \Delta$, and $C_{A}(H)$ for the \textit{centralizer of the normal subobject} $H$ of $A$, that is the largest subobject of $A$ that centralizes $H$, in the sense that $[H,A]=0$ (see \cite{MR2553540}), where $[H,A]$ denotes the categorical commutator in the sense of Huq \cite{Huq}. The following theorem by Cigoli and Mantovani characterizes centralizers and centers in semi-abelian action representable categories. In the present paper we will use this equivalent formulation as the definition of centralizer and center, respectively.

\begin{theorem}\cite{Cigoli} \label{centralizer}
 In any semi-abelian action representable category the {\bf centralizer} $C_{A}(H)$ of a normal subobject $H$ of $A$ is the kernel of the unique morphism $\chi$ induced by the universal property of the split extension classifier $[H]$ of $H$:
\begin{equation}\label{centralizerdiagram}
\xymatrixcolsep{3pc}\xymatrix{
0 \ar[r]^-{} &  H \ar[r]^-{HKer(p_1)} \ar@{=}[d] & Eq(p) \ar@<-0.5ex>[r]_-{p_1} \ar[d]_-{\overline{\chi}} \ar@{}[rd]|<<<{\pullback}  & A \ar[d]^-{\chi} \ar[r]^-{} \ar@<-0.5ex>[l]_-{\Delta} & 0 \\
0 \ar[r]^-{} &  H \ar[r]^-{i_1} & H\rtimes_{\star} [H] \ar@<-0.5ex>[r]_-{p_2}   & [H] \ar[r]^-{} \ar@<-0.5ex>[l]_-{i_2} & 0 }
\end{equation}
\end{theorem}
The {\bf center} $Z(A)$ of an object $A$ is the centralizer of $A$ in $A$:
$$Z(A) =  C_{A}(A).$$
Of course, in the case of the category of groups (resp. of the category of Lie algebras), the upper split extension sequence in Diagram \ref{centralizerdiagram} gives rise to the conjugation action of $A$ on $X$ (resp. to the adjoint representation).  For instance, in the case of groups the action on an element $a\in A $ on an element $x \in H$  is defined by 
$$a \cdot x = \Delta (a) HKer(p_1)(x) \Delta(a)^{-1}= (a,a)(1,x)(a^{-1}, a^{-1})= (1, axa^{-1}),$$
which is exactly the conjugation action, by identifying the element $x$ with its image $HKer(p_1)(x)=(1,x)$ via the kernel map of $p_1$. In this way the classical notion of centralizer is recovered.

\subsection{Centers and centralizers in $\mathbf{Hopf}_{K,coc}$}
From now on $K$ will always denote an algebraically closed field of zero characteristic.
Let $H$ be a normal Hopf subalgebra of a cocommutative Hopf algebra $A$. Then one can construct 
the augmentation ideal $H^+=\{h\in H~|~\epsilon_H(h)=0\}$ of $H$, and
the quotient Hopf algebra $A/J$, where $J=AH^+A$ is the ideal in $A$ generated by $H^+$, which is moreover a Hopf ideal. We write $p:A\to A/J$ for the canonical projection map. 
The pullback $Eq(p)=A\times_H A$ of $p$ along itself can be computed explicitly in the category $\Hopf_{K,coc}$ as follows:
$$Eq(p)=\{a\ot a'\in A\ot A~|~p(a_1)\ot a_2\ot a'=p(a'_1)\ot a\ot a'_2\in A/J\ot A\ot A\}.$$
The projections $p_1,p_2:A\times_HA\to A$ are given respectively by $p_1(a\ot a')=a\epsilon(a')$ and $p_2(a\ot a')=\epsilon(a)a'$ for all $a\ot a'\in A\times_H A$. As explained above, we can then consider the unique morphism $\Delta:A\to A\times_HA$ satisfying $p_1\circ\Delta=id_A=p_2\circ\Delta$, which is exactly given by the corestriction of the comultiplication $A \to A \times A$ of $A$.\\
We know that a kernel of $p_1$ is the morphism 
$HKer(p_1) \colon H \to Eq(p)$ defined by $HKer(p_1)(x)= 1\ot x$, for any $x\in H$. {Since $p_1$ is a split epimorphism in the semi-abelian category $\mathbf{Hopf}_{K,coc}$, it is necessarily the cokernel of its kernel, so that the following sequence is exact}:
\begin{equation}
\xymatrixcolsep{3pc}\xymatrix{
0 \ar[r]^-{} &  H \ar[r]^-{HKer(p_1)}  & Eq(p) \ar@<-0.5ex>[r]_-{p_1} & A \ar[r]^-{} \ar@<-0.5ex>[l]_-{\Delta} & 0.}
\end{equation}
{From Lemma \ref{Molnar} it follows that this split exact sequence is isomorphic to the split exact sequence} in the following commutative diagram
 \begin{equation}
  \xymatrixcolsep{3pc}\xymatrix{
0 \ar[r]^-{} &  H \ar[r]^-{HKer(p_2)} &  H \rtimes A \ar@<-0.5ex>[r]_-{p_2}  & A  \ar[r]^-{} \ar@<-0.5ex>[l]_-{i_2} & 0,}
 \end{equation}
 where the action of $A$ on $H$ is exactly the one recalled in Example \ref{examplesaction}.(3).
 It follows that there is no restriction in replacing diagram \eqref{centralizerdiagram} with the following one
 \begin{equation}\label{centralizerHopf}
\xymatrixcolsep{3pc}\xymatrix{
0 \ar[r]^-{} &  H \ar[r]^-{HKer(p_2)} \ar@{=}[d] &  H \rtimes A \ar@<-0.5ex>[r]_-{p_2} \ar[d]_-{\overline{\chi}} \ar@{}[rd]|<<<{\pullback}  & A \ar[d]^-{\chi} \ar[r]^-{} \ar@<-0.5ex>[l]_-{i_2} & 0 \\
0 \ar[r]^-{} &  H \ar[r]^-{i_1} & H \rtimes_{\star} [H] \ar@<-0.5ex>[r]_-{p_2}   & [H] \ar[r]^-{} \ar@<-0.5ex>[l]_-{i_2} & 0,}
\end{equation}
and the \textit{centralizer} $C_{A}(H)$ in $\mathbf{Hopf}_{K,coc}$ of a normal Hopf subalgebra $H$ of $A$ will be given by the kernel of $\chi$ in the following morphism of split extensions (by Theorem \ref{universal}):

\[\xymatrixcolsep{2pc}\xymatrix{
0 \ar[r]^-{} &  H \ar[r]^-{HKer(p_2)} \ar@{=}[d] & H \rtimes A  \ar@<-0.5ex>[r]_-{p_2} \ar[d]_-{\overline{\chi}}   & A \cong U(L_A)\rtimes_{A} K[G_A] \ar[d]^-{\chi=U(\chi_{L}) \otimes K[\chi_{G}]} \ar[r]^-{} \ar@<-0.5ex>[l]_-{i_2} & 0 \\
0 \ar[r]^-{} &  H \ar[r]^-{i_1} & H\rtimes_{\star} [H] \ar@<-0.5ex>[r]_-{p_2}   & [H]=U\big(\BiDer(H)\big)\rtimes_{\overline{\rho}} K\big[\Aut_{\sf Hopf}(H)\big] \ar[r]^-{} \ar@<-0.5ex>[l]_-{i_2} & 0 }\]
In order to describe the centralizer $C_{A}(H)$ we have to compute the kernel of $\chi$ in the category $\mathbf{Hopf}_{K,coc}$, denoted by $HKer(\chi)$. We are going to show that 
$$HKer(\chi)=HKer\big(U(\chi_{L}) \otimes K[\chi_{G}]\big) \cong HKer\big(U(\chi_{L})\big) \rtimes_{A} HKer\big(K[\chi_{G}]\big).$$ This is proved by computing the kernels of $U(\chi_{L})$, $\chi$ and $K[\chi_{G}]$ in the following commutative diagram
\[\xymatrixcolsep{1pc}\xymatrix{
0 \ar[r]^-{} & U(L_A) \ar[r]^-{i_A} \ar[d]_-{U(\chi_{L})} & U(L_A)\rtimes_{A} K[G_A] \ar@<-0.5ex>[r]_-{p_A} \ar[d]_-{\chi = U(\chi_{L}) \otimes K[\chi_{G}]} & K[G_A] \ar[d]^-{K[\chi_{G}]} \ar[r]^-{} \ar@<-0.5ex>[l]_-{s_A} & 0 \\
0 \ar[r]^-{} &  U\big(\BiDer(H)\big) \ar[r]^-{i_{[H]}} & U\big(\BiDer(H)\big)\rtimes_{\overline{\rho}} K\big[\Aut_{\sf Hopf}(H)\big] \ar@<-0.5ex>[r]_-{p_{[H]}}   & K\big[\Aut_{\sf Hopf}(H)\big] \ar[r]^-{} \ar@<-0.5ex>[l]_-{s_{[H]}} & 0 }\] obtaining the following commutative diagram
\[\xymatrixcolsep{1pc}\xymatrixrowsep{1.5pc}\xymatrix{
 & 0 \ar[d]^-{} & 0 \ar[d]^-{} & 0 \ar[d]^-{} & \\
  & HKer\big(U(\chi_{L})\big) \ar[d]_-{HKer\big(U(\chi_{L})\big)} \ar[r]^-{i} & HKer(\chi) \ar[d]_-{HKer(\chi)} \ar@<-0.5ex>[r]_-{p} & HKer\big(K[\chi_{G}]\big) \ar[d]^-{HKer\big(K[\chi_{G}]\big)} \ar@<-0.5ex>[l]_-{s} &  \\
0 \ar[r]^-{} & U(L_A) \ar[r]^-{i_A} \ar[d]_-{U(\chi_{L})} & U(L_A)\rtimes_{A} K[G_A] \ar@<-0.5ex>[r]_-{p_A} \ar[d]_-{\chi = U(\chi_{L}) \otimes K[\chi_{G}]} & K[G_A] \ar[d]^-{K[\chi_{G}]} \ar[r]^-{} \ar@<-0.5ex>[l]_-{s_A} & 0 \\
0 \ar[r]^-{} &  U\big(\BiDer(H)\big) \ar[r]^-{i_{[H]}} & U\big(\BiDer(H)\big)\rtimes_{\overline{\rho}} K\big[\Aut_{\sf Hopf}(H)\big] \ar@<-0.5ex>[r]_-{p_{[H]}}   & K\big[\Aut_{\sf Hopf}(H)\big] \ar[r]^-{} \ar@<-0.5ex>[l]_-{s_{[H]}} & 0 }\]  where the upper row is easily seen to be a split short exact sequence.
Let us now show that $$HKer\big(K[\chi_{G}]\big) = K\big[ker_{Grp}(\chi_{G})\big],$$ and $$HKer\big(U(\chi_{L})\big) = U\big(ker_{Lie}(\chi_{L})\big).$$
\begin{proposition}\label{noyaugrp}
If $f: G \longrightarrow G'$ is a group homomorphism and $ker(f)$ its kernel in $\mathbf{Grp}$, then the Hopf kernel $HKer(\overline{f})$ of $\overline{f}=K[f]: K[G] \longrightarrow K[G']$ in $\mathbf{Hopf}_{K,coc}$ is given by $K\big[ker(f)\big]$.
\end{proposition}
\begin{proof}
This is an immediate consequence of Proposition \ref{Kadjoint}, since right
adjoint functors between pointed finitely complete categories preserve kernels.
\end{proof}

\begin{proposition}\label{noyaulie}
Let $f: L \longrightarrow L'$ be a Lie $K$-algebra morphism and $ker(f)$ its kernel in the category of Lie $K$-algebras. Then the Hopf kernel $HKer(\overline{f})$ of $\overline{f}=U(f): U(L) \longrightarrow U(L')$ in $\mathbf{Hopf}_{K,coc}$ is given by $U\big(ker(f)\big)$.
\end{proposition}

\begin{proof}
Again this is an immediate consequence of the fact that right
adjoint functors between pointed finitely complete categories preserve kernels, together with Proposition \ref{functorQ} {(one can also verify this property directly, by checking that  $U\big(ker(f)\big) = HKer\big(U(f)\big)$).}
\end{proof}

One can then compute $HKer(\chi)$ by computing $ker_{Lie}(\chi_{L})$ and $ker_{Grp}(\chi_{G})$ and obtain the following:

\begin{theorem}\label{Centralizer}
Given a normal subalgebra $H$ of $A\in \mathbf{Hopf}_{K,coc}$, the centralizer $C_{A}(H)$ in the category of cocommutative Hopf algebras is given by 
$$C_{A}(H) = U\big(ker_{Lie}(\chi_{L})) \rtimes_{A} K\big[ker_{Grp}(\chi_{G})\big]$$
where \begin{center}
$ker_{Lie}(\chi_{L})=\{x\in L_A$ such that $[x,h]=0,$ for all $h\in H\}$
\end{center} and \begin{center}
$ker_{Grp}(\chi_{G})=\{g\in G_A$ such that $gh=hg,$ for all $h\in H\}.$
\end{center}
\end{theorem}

Let us finally mention a result that follows immediately from a more general one in the general context of an action representable semi-abelian category $\textbf{C}$ \cite{Charactersitic}. 
Recall that a subobject $F$ of $G$ in $\textbf{C}$ is called a \textit{characteristic subobject} of $G$ if, whenever $G$ is a normal subobject in $H$ for some object $H$ in {$\textbf C$}, then $F$ is also normal as a subobject in $H$ (see Remark 2.6 in \cite{Charactersitic}).
\begin{proposition}\label{centracat}
For every normal Hopf subalgebra $K$ of $H$, the centralizer $Z_{H}(K)$ of $K$ in $H$ is a normal Hopf subalgebra in $H$.  Moreover, the center $Z(H)$ is a characteristic Hopf subalgebra of $H$.
\end{proposition}

\subsection{Arbitrary Hopf algebra centers}

N.\ Andruskiewitsch \cite{And:center} 
introduced the following definition of the center for an arbitrary (not necessarily cocommutative) Hopf algebra 
(see also \cite{Center}). 

\begin{definition}
Let $A$ be any Hopf algebra. The \textit{Hopf algebra center} $HZ(A)$ of $A$ is the largest Hopf subalgebra of $A$ contained in the algebraic center $Z_{alg}(A)$ of $A$:
$$Z_{alg}(A) = \{x\in A \hspace*{0.2cm} s.t. \hspace*{0.2cm} xy=yx \hspace*{0.2cm} \forall y\in A\} .$$
\end{definition}
A.\ Chirvasitu and P.\ Kasprzak observe in \cite[Theorem 2.2]{Center} that in the case of a Hopf algebra $A$ with bijective antipode (which holds for example in the cocommutative context), the Hopf algebra center of $A$ can be described  as:
$$HZ(A)= \{x\in A~|~ \Delta(x)\in A\otimes Z_{alg}(A)\}$$

Our final result shows that the Hopf algebra center in the sense of Andruskiewitsch coincides with the categorical notion of a center for cocommutative Hopf algebras over an algebraically closed field of characteristic zero.

\begin{proposition}
Given a cocommutative Hopf algebra $A$, then 
$$Z(A)=HZ(A).$$
Here, $Z(A)=HKer(\chi)$ denotes the categorical center of $A$, where $\chi\colon A \longrightarrow [A]$ is the unique morphism associated to the action of $A$ on itself by conjugation. Writing $\chi := U(\chi_{L}) \otimes K[\chi_{G}]$, we have $$Z(A) = U\big(ker_{Lie}(\chi_{L})\big) \rtimes_{A} K\big[ker_{Grp}(\chi_{G})\big]$$ where 
\begin{center}
$ker_{Lie}(\chi_{L})=\{x\in L_A$ such that $[x,a]=0,$ for all $a\in A\}$
\end{center} 
and 
\begin{center}
$ker_{Grp}(\chi_{G})=\{g\in G_A$ such that $ga=ag,$ for all $a\in A\}.$
\end{center}
\end{proposition}

\begin{proof}
The explicit description of the categorical center $Z(A)$ follows directly from Theorem \ref{Centralizer}.

Let us first prove that $Z(A) \subseteq HZ(A)$. An element $x\in A$ is in $Z(A)=HKer(\chi)$ if and only if $\chi(x_1)\ot x_2= id_{[A]}\ot x \in [A]\ot A$. Let us prove that $x\in Z_{alg}(A)$, i.e. $xy=yx$, $\forall y\in A$. To this end, denote the action of $[A]$ on $A$ by $$\overline{\rho}: [A]\ot A\to A,\ \overline{\rho}(u\ot x)=u.x$$ and define the following map for any $y\in A$ 
$$r_y: [A]\ot A\to A,\ r_{y}(u\ot x)= (u.y)x.$$ 
We have $r_{y}(id_{[A]}\ot x) = id_{A}(y)x=yx$, and
\begin{eqnarray*}
r_y\big(\chi(x_1)\ot x_2\big) &=& \chi(x_{1})(y)x_{2}\\
 &=& x_{1}yS(x_{2})x_{3} \\
&=&x_{1}y \epsilon(x_{2}) \\
&=& xy.
\end{eqnarray*}
We obtain that $Z(A)\subset Z_{alg}(A)$. Since $Z(A)$ is a Hopf subalgebra of $A$, and $HZ(A)$ is the largest Hopf subalgebra in $Z_{alg}(A)$, we obtain that $Z(A) \subseteq HZ(A)$.

On the other hand, taking any $a\in HZ(A)\subset Z_{alg}(A)$, we find for any other $b\in A$ that the conjugate action of $a$ on $b$ is trivial:
$$a\cdot b=a_{1}bS(a_{2})=ba_{1}S(a_{2})=b\epsilon(a)$$
Consequently, we find that $\chi$ composed with the inclusion map $HZ(A)\subset A$ is the zero morphism in $\Hopf_{K,coc}$. By the universal property of the kernel $HKer(\chi)=Z(A)$ we find then that 
$HZ(A)\subset Z(A)$ and both have to be equal.
\end{proof}


\begin{thebibliography}{10}
\providecommand{\url}[1]{{#1}}
\providecommand{\urlprefix}{URL }
\expandafter\ifx\csname urlstyle\endcsname\relax
  \providecommand{\doi}[1]{DOI~\discretionary{}{}{}#1}\else
  \providecommand{\doi}{DOI~\discretionary{}{}{}\begingroup
  \urlstyle{rm}\Url}\fi

\bibitem{And:center}
Andruskiewitsch, N.: Notes on extensions of {H}opf algebras.
\newblock Canad. J. Math. \textbf{48}(1), 3--42 (1996).

\bibitem{Barr3}
Barr, M.: Coalgebras over a commutative ring.
\newblock Journal of {A}lgebra \textbf{32}, 600--610 (1974).

\bibitem{classi}
Borceux, F., Bourn, D.: Split extension classifier and centrality.
\newblock Contemporary Mathematics \textbf{431} (2007).

\bibitem{Clementino}
Borceux, F., Clementino, M., Montoli, A.: On the representability of actions
  for topological algebras.
\newblock Mathematical Texts (University of Coimbra) (46), 41--66 (2014).

\bibitem{BJK2}
Borceux, F., Janelidze, G., Kelly, G.: On the representability of actions in a
  semi-abelian category.
\newblock Theory Appl. Categories \textbf{14\hspace*{1mm}}(11), 244 -- 286
  (2005).

\bibitem{BJK}
Borceux, F., Janelidze, G., Kelly, G.M.: Internal object actions.
\newblock Commentationes Mathematicae Universitatis Carolinae \textbf{46}(2),
  235 -- 255 (2005).

\bibitem{protomodularity}
Bourn, D.: Normalization equivalence, kernel equivalence and affine categories.
\newblock Springer Berlin Heidelberg, Lecture Notes in Mathematics
  \textbf{1488}, 43--62 (1991).

\bibitem{MR2553540}
Bourn, D., Janelidze, G.: Centralizers in action accessible categories.
\newblock Cah. Topol. G\'eom. Diff\'er. Cat\'eg. \textbf{50}(3), 211--232
  (2009).
  
\bibitem{Burris}  
Burris, S., Sankappanavar, H.\ P.:
A course in universal algebra.
\newblock Graduate Texts in Mathematics, \textbf{78}. Springer-Verlag, New York-Berlin, 1981.

\bibitem{Center}
Chirvasitu, A., Kasprzak, P.: On the {H}opf (co)center of a {H}opf algebra.
\newblock Journal of Algebra \textbf{464}, 141 -- 174 (2016).

\bibitem{Cigoli}
Cigoli, A.S., Mantovani, S.: Action accessibility via centralizers.
\newblock Journal of Pure and Applied Algebra \textbf{216}(8-9), 1852 -- 1865
  (2012).
\newblock Special Issue devoted to the International Conference in Category
  Theory `CT2010'.

\bibitem{Charactersitic}
Cigoli, A.S., Montoli, A.: Characteristic subobjects in semi-abelian
  categories.
\newblock Theory and {A}pplications of {C}ategories \textbf{30}(7), 206 -- 228
  (2015).

\bibitem{GMVdL}
Garc\'ia-Mart\'inez, X., {Van der Linden}, T.: A note on split extensions of
  bialgebras.
\newblock arXiv:1701.00665  (2017).

\bibitem{GKV}
Gran, M., Kadjo, G., Vercruysse, J.: A torsion theory in the category of
  cocommutative {H}opf algebras.
\newblock Applied Categorical Structures \textbf{24}(3), 269--282 (2016).

\bibitem{Huq}
Huq, S.A.: Commutator, nilpotency, and solvability in categories.
\newblock Quart. J. Math. Oxford Ser. (2) \textbf{19}, 363--389 (1968).

\bibitem{JMT} Janelidze G., M\'arki L. and Tholen W.: Semi-abelian categories, \textbf{168}, 367--386 (2002).

\bibitem{Manin}
Manin, Yu. I.: Quantum groups and noncommutative geometry, Universit\'e de Montr\'eal, Centre de Recherches Math\'ematiques, Montreal, QC, 1988.

\bibitem{Milnor}
Milnor, J., Moore, J.: On the structure of {H}opf algebras.
\newblock Annals of Mathematics, Second Series \textbf{81\hspace*{1mm}}(2),
  211--264 (1965).

\bibitem{Molnar}
Molnar, R.K.: Semi-direct products of {H}opf algebras.
\newblock J. Algebra \textbf{47}, 29--51 (1977).

\bibitem{Montgom}
Montgomery, S.: Hopf algebras and their actions on rings.
\newblock CBMS Regional Conference Series in Mathematics \textbf{82} (1993).

\bibitem{Sweedler}
Sweedler, M.E.: Hopf algebras.
\newblock Benjamin New York  (1969).

\bibitem{Tambara}
Tambara, C.: The coendomorphism bialgebra of an algebra, J. Fac. Sci. Univ. Tokyo Sect. IA Math. \textbf{37}, 425--456 (1990).


\bibitem{Vespa}
Vespa, C., Wambst, M.: On some properties of the category of cocommutative
  {H}opf algebras.
\newblock arXiv:1502.04001  (2015).

\end{thebibliography}
\end{document}